\documentclass[12pt]{article}

\usepackage{geometry}
\usepackage{hyperref} 

\usepackage{amsmath}
\usepackage{amssymb}
\usepackage{amscd}
\usepackage{textcomp}
\usepackage{dsfont}
\usepackage{mathrsfs}

\long\def\forget#1{}

\usepackage{color}
\newcounter{commentcounter}
\newcommand{\comment}[1]{\stepcounter{commentcounter}{\color{red}\textbf{Comment \arabic{commentcounter}.} #1}
		\immediate\write16{}
		\immediate\write16{Warning: There was still a comment . . . }
		\immediate\write16{}}
	
	\newcounter{urscommentcounter}

	\def\?{\ 
		{\bf\color{red}???}\ 
		\immediate\write16{}
		\immediate\write16{Warning: There was still a question mark . . . }
		\immediate\write16{}}

	\usepackage{amsthm}
	\theoremstyle{plain}
	\newtheorem{theorem}{Theorem}[section]
	
	\newtheorem{lemma}[theorem]{Lemma}

	\newtheorem{corollary}[theorem]{Corollary}
	\newtheorem{proposition}[theorem]{Proposition}

	\theoremstyle{definition}
	\newtheorem{definition}[theorem]{Definition}
	\newtheorem{definition-theorem}[theorem]{Definition-Theorem}
	\newtheorem{definition-remark}[theorem]{Definition-Remark}
	\newtheorem{point}[theorem]{}
	
	\newtheorem{remark}[theorem]{Remark}
	
	\theoremstyle{remark}

	\usepackage{xy}
	\xyoption{all}
	
	\newdir^{ (}{{}*!/-3pt/\dir^{(}}    
	\newdir^{  }{{}*!/-3pt/\dir^{}}    
	\newdir_{ (}{{}*!/-3pt/\dir_{(}}    
	\newdir_{  }{{}*!/-3pt/\dir_{}}

	\newcounter{zahl}

	\def\theenumi{(\alph{enumi})}

	\def\p@enumii{\theenumi}

	\newcommand{\DS}{\displaystyle}
	\newcommand{\TS}{\textstyle}
	\newcommand{\SC}{\scriptstyle}
	\newcommand{\SSC}{\scriptscriptstyle}

	\newcommand{\cG}{\mathcal{G}}

	\newcommand{\cM}{\mathcal{M}}
	\newcommand{\cO}{\mathcal{O}}

	\DeclareMathOperator{\Aut}{Aut}

	\DeclareMathOperator{\Frob}{Frob}
	\DeclareMathOperator{\Gal}{Gal}
	\DeclareMathOperator{\GL}{GL}

	\DeclareMathOperator{\Hom}{Hom}

	\DeclareMathOperator{\Isom}{Isom}

	\DeclareMathOperator{\QHom}{QHom}
	\DeclareMathOperator{\Quot}{Frac}
	
	\DeclareMathOperator{\Rep}{Rep}

	\DeclareMathOperator{\SL}{SL}
	
	\DeclareMathOperator{\Spec}{Spec}
	\DeclareMathOperator{\Spf}{Spf}

	\DeclareMathOperator{\Var}{V}

	\newcommand{\alg}{{\rm alg}}

	\newcommand{\et}{{\acute{e}t\/}}
	\newcommand{\fppf}{{\it fppf\/}}
	\newcommand{\fpqc}{{\it fpqc\/}}

	\DeclareMathOperator{\id}{\,id}
	\DeclareMathOperator{\im}{im}

	\newcommand{\mot}{{\cM ot_C^{\ul \nu}}}

	\newcommand{\red}{{\rm red}}

	\newcommand{\sep}{{\rm sep}}

	\DeclareMathOperator{\whtimes}{\mathchoice
		{\wh{\raisebox{0ex}[0ex]{$\DS\times$}}}
		{\wh{\raisebox{0ex}[0ex]{$\TS\times$}}}
		{\wh{\raisebox{0ex}[0ex]{$\SC\times$}}}
		{\wh{\raisebox{0ex}[0ex]{$\SSC\times$}}}}

	\renewcommand{\phi}{\varphi}
	\renewcommand{\epsilon}{\varepsilon}

	\usepackage{amsfonts}
	\newcommand{\BOne} {{\mathchoice{\hbox{\rm1\kern-2.7pt l\kern.9pt}}
			{\hbox{\rm1\kern-2.7pt l\kern.9pt}}
			{\hbox{\scriptsize\rm1\kern-2.3pt l\kern.4pt}}
			{\hbox{\scriptsize\rm1\kern-2.4pt l\kern.5pt}}}}
	
	\newcommand{\BA}{{\mathbb{A}}}

	\newcommand{\BD}{{\mathbb{D}}}
	
	\newcommand{\BF}{{\mathbb{F}}}

	\newcommand{\BL}{{\mathbb{L}}}
	
	\newcommand{\BN}{{\mathbb{N}}}
	\newcommand{\BO}{{\mathbb{O}}}
	\newcommand{\BP}{{\mathbb{P}}}
	\newcommand{\BQ}{{\mathbb{Q}}}

	\newcommand{\BZ}{{\mathbb{Z}}}

	\newcommand{\CF}{{\cal{F}}}
	\newcommand{\CG}{{\cal{G}}}
	\newcommand{\CH}{{\cal{H}}}
	
	\newcommand{\CJ}{{\cal{J}}}
	
	\newcommand{\CL}{{\cal{L}}}
	\newcommand{\CM}{{\cal{M}}}
	\newcommand{\CN}{{\cal{N}}}
	\newcommand{\CO}{{\cal{O}}}
	\newcommand{\CP}{{\cal{P}}}

	\newcommand{\CS}{{\cal{S}}}
	\newcommand{\CT}{{\cal{T}}}
	
	\newcommand{\CV}{{\cal{V}}}
	
	\newcommand{\CX}{{\cal{X}}}
	\newcommand{\CY}{{\cal{Y}}}
	\newcommand{\CZ}{{\cal{Z}}}

	\newcommand{\FG}{{\mathfrak{G}}}

	\newcommand{\FM}{{\mathfrak{M}}}

	\newcommand{\FP}{{\mathfrak{P}}}

	\newcommand{\FT}{{\mathfrak{T}}}

	\newcommand{\iii}{j}
	
	\newcommand{\scrA}{{\mathscr{A}}}
	\newcommand{\scrB}{{\mathscr{B}}}

	\newcommand{\scrH}{{\mathscr{H}}}

	\newcommand{\scrS}{{\mathscr{S}}}

	\let\setminus\smallsetminus
	
	\newcommand{\es}{\enspace}

	\newcommand{\ul}[1]{{\underline{#1}}}
	\newcommand{\ol}[1]{{\overline{#1}}}
	\newcommand{\wh}[1]{{\widehat{#1}}}
	\newcommand{\wt}[1]{{\widetilde{#1}}}
	
	\usepackage{ifthen}
	
	\newcommand{\invlim}[1][]{\ifthenelse{\equal{#1}{}}
		{\DS \lim_{\longleftarrow}}
		{\DS \lim_{\underset{#1}{\longleftarrow}}}
	}
	
	\newcommand{\dirlim}[1][]{\ifthenelse{\equal{#1}{}}
		{\DS \lim_{\longrightarrow}}
		{\DS \lim_{\underset{#1}{\longrightarrow}}}
	}

	\newcommand{\dbl}{{\mathchoice{\mbox{\rm [\hspace{-0.15em}[}}
			{\mbox{\rm [\hspace{-0.15em}[}}
			{\mbox{\scriptsize\rm [\hspace{-0.15em}[}}
			{\mbox{\tiny\rm [\hspace{-0.15em}[}}}}
	\newcommand{\dbr}{{\mathchoice{\mbox{\rm ]\hspace{-0.15em}]}}
			{\mbox{\rm ]\hspace{-0.15em}]}}
			{\mbox{\scriptsize\rm ]\hspace{-0.15em}]}}
			{\mbox{\tiny\rm ]\hspace{-0.15em}]}}}}
	\newcommand{\dpl}{{\mathchoice{\mbox{\rm (\hspace{-0.15em}(}}
			{\mbox{\rm (\hspace{-0.15em}(}}
			{\mbox{\scriptsize\rm (\hspace{-0.15em}(}}
			{\mbox{\tiny\rm (\hspace{-0.15em}(}}}}
	\newcommand{\dpr}{{\mathchoice{\mbox{\rm )\hspace{-0.15em})}}
			{\mbox{\rm )\hspace{-0.15em})}}
			{\mbox{\scriptsize\rm )\hspace{-0.15em})}}
			{\mbox{\tiny\rm )\hspace{-0.15em})}}}}

	\newcommand{\dotBD}{\vbox{\hbox{\kern2pt\bf.}\vskip-4.5pt\hbox{$\BD$}}}

	\DeclareMathOperator{\QIsog}{QIsog}
	\DeclareMathOperator{\Nilp}{\CN \!{\it ilp}}
	
	\DeclareMathOperator{\Sets}{\CS \!{\it ets}}

	\def\s{\sigma^\ast}

	\def\ulHZ{{\underline{\hat Z\!}\,}{}}

	\def\longto{\longrightarrow}
	\def\into{\hookrightarrow}

	\def\isoto{\stackrel{}{\mbox{\hspace{1mm}\raisebox{+1.4mm}{$\SC\sim$}\hspace{-3.5mm}$\longrightarrow$}}}

	\newbox\mybox
	\def\arrover#1{\mathrel{
			\setbox\mybox=\hbox spread 1.4em{\hfil$\scriptstyle#1$\hfil}
			\vbox{\offinterlineskip\copy\mybox
				\hbox to\wd\mybox{\rightarrowfill}}}}
	
	\newcommand{\ppsi}{\delta}
	
	\newcommand{\RZ}{\ul{\CM}_{\ul{\BL}_0}^{\hat{Z}}}

	\newcommand{\BaseOfD}{\BF}

	\newcommand{\BaseFldOfLocSht}{k}
	
	\newcommand{\BaseFldInSectUnif}{k}

	\newcommand{\genericG}{P}
	
	\newcommand{\Sht}{Sht}
	
	\DeclareMathOperator{\SpaceFl}{\CF\ell}
	\newcommand{\tauGlob}{\tau}
	\newcommand{\tauLoc}{\hat\tau}
	\newcommand{\charsect}{s}

	\newcommand{\SSS}{S}
	\newcommand{\TTT}{T}

\begin{document}
		
		\author{Esmail Arasteh Rad and Urs Hartl}

		\date{\today}

		\title{The Conjecture of Langlands and Rapoport for the Moduli Stacks of $\FG$-Shtukas\\}
		
		\maketitle
		
		\begin{abstract}

			In this article we formulate and prove the analogue of the Langlands-Rapoport conjecture for the moduli stacks of global $\FG$-shtukas. Here $\FG$ is a parahoric Bruhat-Tits group scheme over a smooth projective curve $C$  over a finite field $\BF_q$. 
			\noindent
			
			\noindent
			{\it Mathematics Subject Classification (2000)\/}: 
			11G09,  
			(11G18,  
			14L05)  
		\end{abstract}
		
		\tableofcontents 
		
		\section{Introduction}

		According to Langlands' philosophy it is highly desirable to understand the cohomology of Shimura varieties, as they carry Hecke and Galois symmetries that can be used to relate Automorphic representations to Galois representations. This correspondence is supposed to be canonical, and this brings one to study the zeta functions associated to these varieties. Concerning this, Langlands made a conjecture about the structure of mod $p$ points of Shimura varieties, from which the expression of the zeta function of certain Shimura varieties as a product of automorphic $L$-functions follows. The conjecture was made more precise by Kottwitz \cite{Ko2} and consequently refined by Langlands and Rapoport \cite{LR}. This conjecture provides a conceptual group theoretic description of mod $p$ points of Shimura varieties.  As an evidence for the conjecture, Langlands and Rapoport proved the conjecture for simple Shimura varieties of PEL types A and C, under the assumption of Grothendieck's standard conjectures and the Tate conjecture for varieties over finite fields, and the Hodge conjecture for abelian varieties. Following work of many authors, including Morita, Milne, Pfau, Deligne, Reimann, Zink, Kottwitz \cite{Ko3} proved the conjecture for PEL Shimura varieties of type A and C and according to technical issues arising in the general set up,  Kisin, considerably later, proved the conjecture for Shimura varieties of \emph{abelian type} with \emph{hyperspecial} level structure \cite{Kis}. Recall that the Shimura varieties of abelian type are exactly those which fulfill the Deligne conception of Shimura varieties \cite{De3}. Indeed, concerning the issue of existence of canonical integral model for Shimura varieties, this assumption seems to be essential even for stating the conjecture.\\

		In this article, as an application of our results related to the local theory of global $\FG$-shtukas, which we developed in \cite{AH_Local} and \cite{AH_Global}, and also the motivic interpretation of these moduli stacks, see \cite{CMot}, we formulate the Langlands-Rapoport conjecture for the moduli stacks of global $\FG$-shtukas and prove it in a relatively general set up. On the other hand, in \cite{AH_LM} and \cite{Ara}, the authors studied the theory of local models for these moduli stacks, and their local counterparts. This is a crucial step towards computing the semi-simple trace of Frobenius on the cohomology of these moduli stacks, as it relates the semi-simple trace on the moduli of $\FG$-shtukas to the semi-simple trace on the Schubert varieties inside affine flag varieties, e.g. see \cite[Subsection 4.3]{Ara}, where one can implement the test-function conjecture of Kottwitz, see \cite{H-R2}, in order to express it in terms of certain functions in the center of the Hecke algebra. Consequently, these results provide some of the main ingredients for computing the Hasse-Weil zeta-function for the moduli stacks of global $\FG$-shtukas.

		Let $C$ be a smooth projective, geometrically irreducible curve over a finite field $\BF_q$ with $q$ elements, and let $Q=\BF_q(C)$ be its function field. Fix an $n$-tuple $\ul\nu=(\nu_1,\ldots,\nu_n)$ of pairwise different closed points $\nu_i\in C$. In the situation of Shimura varieties, $\BZ$ and $p$ are the analogue of our $C$ and $\ul\nu$. Consider a \emph{parahoric} (\emph{Bruhat-Tits}) \emph{group scheme} $\FG$ over $C$ in the sense of \cite[D\'efinition 5.2.6]{B-T} and \cite{H-R1}. According to the analogy between function fields and number fields, the moduli stacks of global $\FG$-shtukas appear as the function field counterparts of (the canonical integral model for) Shimura varieties. \\

		According to the analogy with the theory of Shimura varieties, which we mentioned above, one defines a $\nabla\scrH$-data as a pair $(\FG, \hat{Z}_\ul\nu)$ consisting of a parahoric group scheme $\FG$ over the curve $C$ and an $n$-tuple $\hat{Z}_\ul\nu=(\hat{Z}_{\nu_i})_i$ of bounds $\hat{Z}_{\nu_i}$.  Roughly speaking a bound $\hat{Z}_{\nu_i}$ is a closed subscheme of a certain formal completion $\wh{\CF\ell}_{\BP_{\nu_i}}$ of the twisted affine flag variety $\CF\ell_{\BP_{\nu_i}}$, associated with the parahoric group $\BP_{\nu_i}$; see definition \ref{DefBDLocal}. Here $\BP_{\nu_i}$ denote the parahoric group obtained by base change of $\FG$ to $\Spec A_{\nu_i}$, the spectrum of the completion of the stalk $\CO_{C,\nu_i}$ at $\nu_i$. To such a pair $(\FG, \hat{Z}_\ul\nu)$ one may associate a tower $\nabla_n^{\ast, \hat{Z}_\ul\nu}\scrH^1(C,\FG)^{\ul\nu}$ of formal algebraic stacks $\nabla_n^{H, \hat{Z}_\ul\nu}\scrH^1(C,\FG)^{\ul\nu}$, which are moduli spaces for global $\FG$-shtukas with fixed characteristics $\ul\nu$ (also called paws by V. Lafforgue \cite{Laff12}), bounded by $\hat{Z}_\ul\nu$, and endowed with $H$-level structure for a compact open subgroup $H\subseteq \FG(\BA^\ul\nu)$. Note that the tower of formal algebraic stacks lies over a reflex ring $R:=R(\FG,\hat{Z}_{\ul\nu})$ with finite residue field $\kappa:=\kappa(\FG,\hat{Z}_\ul\nu)$. It naturally carries an action of the adelic group $\FG(\BA^\ul\nu)$, that operates through Hecke correspondences. These formal algebraic stacks may be viewed as the function field counterparts of integral models for Shimura varieties and the parahoric case may reflect the bad reduction at the characteristic place. To such $\nabla\scrH$-data, we associate a groupoid 
		
		$$
		\CL^\ast(\FG, \hat{Z}_\ul\nu):=\bigsqcup_{\phi}\scrS(\phi),
		$$
		where the index $\phi$ runs over isomorphism classes of homomorphisms
		\begin{equation*}
			\varphi : \mathfrak{P} \longrightarrow \FG  \times_C  \Spec\bigl( \breve{Q} \otimes_{Q} \breve{Q}\bigr)
		\end{equation*}
		of groupoids. Here $\breve{Q}=Q\otimes_{\BF_q}\BF_q^{\alg}$ and $\FP$ is the Tannakian fundamental groupoid associated to a function fields motivic category $\CM ot_C^\ul\nu(\BF_q^\alg)$, introduced and studied in \cite{CMot}, see Definition \ref{DefCatCMotives} and Theorem \ref{ThmGroupTheoreticDescriptionOfIndexSet} below. Note that each groupoid $\scrS(\phi)$ is naturally equipped with an action of the adelic group $\FG(\BA^\ul\nu)$ which again operates through Hecke correspondences, with the action of the Frobenius $\Phi$ in $Gal(\ol\BF/\kappa)$, and with the action of $Z(Q)$ where $Z\subset \FG$ is the center; see Definition~\ref{DefFunctor_L} and Proposition~\ref{PropThetaisEquivariantUnderGZPhi}. Here $\ol\BF$ is an algebraic closure of the finite field $\kappa$.

		\noindent
		Here we state the analog of the Langlands-Rapoport conjecture for moduli of $\FG$-shtukas. 
		
		\begin{theorem}\label{ThmL-RForG-ShtINT}
			
			There exist a canonical $\FG(\BA^{\ul\nu})\times Z(Q)$-equivariant isomorphism of functors
			
			\begin{equation}\label{EqL-Risom}
				\CL^*(-)\tilde{\to} \nabla_n^{*,-}\scrH^1(C,-)^{\ul\nu}(\ol \BF).
			\end{equation}

			Moreover via this isomorphism, the operation $\ul\Phi$ on the left hand side of the above isomorphism corresponds to the Frobenius endomorphism $\sigma$ on the right hand side. Furthermore:\\
			
			\begin{enumerate}
				\item 
				The functor $\CL^\ast(-)$ can be replaced by $\CL_{adm}^\ast(-)$. Note that for tamely ramified $\FG$ there exists a group theoretic description for $\CL_{adm}^\ast(-)$, moreover, \forget{due to Xuhua He when $X_{\underline{v}} ({\cal{G}}') \not= \emptyset$}
				
				\item
				when $\FG$ is quasi-split, one may further replace $\CL^\ast(-)$ by $\CL_{spe}^\ast(-)$.
				
			\end{enumerate}

		\end{theorem}
		
		This is Theorem \ref{ThmL-RForG-ShtTXT} in the text. For the definition of the functors $\CL_{adm}^\ast(-)$ and $\CL_{spe}^\ast(-)$ see Definition \ref{DefFunctor_L}. The proof is mainly given by adjoining the results obtained in \cite{AH_Local}, \cite{AH_Global} and \cite{CMot}, together with the main result of He in \cite{He}, and Hamacher and Kim in \cite{H-K}, see Section \ref{Sect_Proof_Of_LR_Conj}.

		\subsection{Notation and Conventions}\label{SubsectNotationandConventions}

		\begin{tabbing}
			$\genericG_\nu:=\FG\times_C\Spec Q_\nu,$\; \=\kill

			$C$\> \parbox[t]{0.711\textwidth}{a smooth projective geometrically irreducible curve over $\BF_q$,}\\[1mm]
			$Q:=\BF_q(C)$\> the function field of $C$,\\[1mm]
			
			$\breve{Q}$\>$Q\otimes_{\BF_q}\BF_q^{\alg}$\\[1mm]
			$\nu$\> a closed point of $C$, also called a \emph{place} of $C$,\\[1mm]
			$\BF_\nu$\> the residue field at the place $\nu$ on $C$,\\[1mm]
			$\BaseOfD$\> a finite field containing $\BF_q$,\\[1mm]

			$A_\nu$\> the completion of the stalk $\CO_{C,\nu}$ at $\nu$,\\[1mm]
			$Q_\nu:=\Quot(A_\nu)$\> its fraction field,\\[1mm]
			
			
			
			$\BO^\ul\nu$\> the ring of integral adeles of $C$ outside $\ul\nu$,\\[1mm]
			
			$\BA_Q^{\ul\nu}:=\BO^\ul\nu\otimes_{\CO_C} Q$\> the ring of adeles of $C$ outside $\ul\nu$,\\[1mm]

			$\BD_R:=\Spec R\dbl z \dbr$ \> \parbox[t]{0.711\textwidth}{the spectrum of the ring of formal power series in $z$ with coefficients in an $\BaseOfD$-algebra $R$,}\\[1mm]
			$\hat{\BD}_R:=\Spf R\dbl z \dbr$ \>\parbox[t]{0.711\textwidth}{ the formal spectrum of $R\dbl z\dbr$ with respect to the $z$-adic topology.}\\[1mm]
			
			\noindent
			When $R= \BaseOfD$ we drop the subscript $R$ from the notation of $\BD_R$ and $\hat{\BD}_R$.\\[1mm]

			$\FG$\> \parbox[t]{0.711\textwidth}{\forget{a flat affine group scheme of finite type over $C$} a parahoric (Bruhat-Tits) group scheme over $C$; recall that} \\[1mm]
			
		\end{tabbing}
		
		\begin{definition}\label{DefParahoric}
			A smooth affine group scheme $\FG$ over $C$ is called a \emph{parahoric} (\emph{Bruhat-Tits}) \emph{group scheme} if
			\begin{enumerate}
				\item
				all geometric fibers of $\FG$ are connected and the generic fiber of $\FG$ is reductive over $\BF_q(C)$,
				\item 
				for any ramification point $\nu$ of $\FG$ (i.e. those points $\nu$ of $C$, for which the fiber above $\nu$ is not reductive) the group scheme $\BP_\nu :=\FG\times_C\Spec A_\nu$ is a parahoric group scheme over $A_\nu$, as defined by Bruhat and Tits \cite[D\'efinition~5.2.6]{B-T}; see also \cite{H-R1}.
			\end{enumerate}
		\end{definition}

		\begin{tabbing}
			$\genericG_\nu:=\FG\times_C\Spec Q_\nu,$\; \=\kill
			
			$G$\>\parbox[t]{0.711\textwidth}{ the generic fiber of $\FG$ over $Q$,},\\[1mm]
			
			$\genericG_\nu:=\FG\times_C\Spec Q_\nu$\> the generic fiber of $\BP_\nu$ over $\Spec Q_\nu$,

		\end{tabbing}

		\noindent
		For a formal scheme $\wh S$ we denote by $\Nilp_{\wh S}$ the category of schemes over $\wh S$ on which an ideal of definition of $\wh S$ is locally nilpotent. We  equip $\Nilp_{\wh S}$ with the \fppf-topology. We also denote by
		\begin{tabbing}
			$\genericG_\nu:=\FG\times_C\Spec Q_\nu,$\; \=\kill
			$\ul \nu:=(\nu_i)_{i=1\ldots n}$\> an $n$-tuple of closed points of $C$,\\[1mm]
			$A_{\ul\nu}$\> \parbox[t]{0.711\textwidth}{the completion of the local ring $\CO_{C^n,\ul\nu}$ of $C^n$ at the closed point $\ul\nu=(\nu_i)$,}\\[1mm]
			$\Nilp_{A_{\ul\nu}}:=\Nilp_{\Spf A_{\ul\nu}}$\> \parbox[t]{0.711\textwidth}{the category of schemes over $C^n$ on which the ideal defining the closed point $\ul\nu\in C^n$ is locally nilpotent,}\\[2mm]
			$\Nilp_{\BaseOfD\dbl\zeta\dbr}:=\Nilp_{\hat\BD}$\> \parbox[t]{0.711\textwidth}{the category of $\BD$-schemes $S$ for which the image of $z$ in $\CO_S$ is locally nilpotent. We denote the image of $z$ by $\zeta$ since we need to distinguish it from $z\in\CO_\BD$.}\\[2mm]
			
			$\BP$\> \parbox[t]{0.711\textwidth}{\forget{a flat affine group scheme of finite type over $\BD=\Spec \wh A$}a smooth affine group scheme of finite type over $\BD=\Spec \BF\dbl z\dbr$,}\\[1mm] 
			$\genericG:=\BP\times_{\BD}\Spec Q$\> the generic fiber of $\BP$ over $\Spec Q$.

		\end{tabbing}


		We let $\Rep_{\BO^{\ul\nu}} \FG$ be the category of representations of $\FG$ in finite free $\BO^{\ul\nu}$-modules $V$. More precisely, $\Rep_{\BO^{\ul\nu}} \FG$ is the category of $\BO^{\ul\nu}\,$-morphisms \mbox{$\rho\colon\FG\times_C\Spec\BO^{\ul\nu}\to\GL_{\BO^{\ul\nu}}(V)$}.\\

		\noindent
		Let $S$ be an $\BF_q$-scheme. We denote by $\sigma_S :  S \to S$ its $\BF_q$-Frobenius endomorphism which acts as the identity on the points of $S$ and as the $q$-power map on the structure sheaf. Likewise we let $\hat{\sigma}_S: S\to S$ be the $\BaseOfD$-Frobenius endomorphism of an $\BaseOfD$-scheme $S$. We set
		\[
		C_S := C \times_{\Spec\BF_q} S\quad \text{ and }\quad\sigma := \id_C \times \sigma_S.
		\]
		Let $H$ be a sheaf of groups (for the \fppf-topology) on a scheme $X$. In this article a (\emph{right}) \emph{$H$-torsor} (also called an \emph{$H$-bundle}) on $X$ is a sheaf $\CG$ for the \fppf-topology on $X$ together with a (right) action of the sheaf $H$ such that $\CG$ is isomorphic to $H$ on an \fppf-covering of $X$. Here $H$ is viewed as an $H$-torsor by right multiplication.

		\bigskip

		We denote by $\breve{Q}_\nu$ the completion of the maximal unramified extension of $Q_\nu$ in an algebraic closure of $Q_\nu$, and $\breve{A}_\nu$, its ring of integers.

		\begin{definition}\label{DefLoopGps}
			The \emph{group of positive loops associated with $\BP$} is the affine group scheme $L^+\BP$ over $\BaseOfD$ whose $R$-valued points for an $\BaseOfD$-algebra $R$ are 
			\[
			L^+\BP(R):=\BP(R\dbl z\dbr):=\BP(\BD_R):=\Hom_\BD(\BD_R,\BP)\,.
			\]
			The \emph{group of loops associated with $\BP$} is the $\fpqc$-sheaf of groups $LP$ over $\BaseOfD$ whose $R$-valued points for an $\BaseOfD$-algebra $R$ are 
			\[
			LP(R):=P(R\dpl z\dpr):=P(\dot{\BD}_R):=\Hom_{\dot\BD}(\dot\BD_R,P)\,,
			\]
			where we write $R\dpl z\dpr:=R\dbl z \dbr[\frac{1}{z}]$ and $\dot{\BD}_R:=\Spec R\dpl z\dpr$. It is representable by an ind-scheme of ind-finite type over $\BaseOfD$; see \cite[\S\,1.a]{PR2}, or \cite[\S4.5]{B-D}.
			Let $\scrH^1(\Spec \BaseOfD,L^+\BP)\,:=\,[\Spec \BaseOfD/L^+\BP]$ (respectively $\scrH^1(\Spec \BaseOfD,P)\,:=\,[\Spec \BaseOfD/LP]$) denote the classifying space of $L^+\BP$-torsors (respectively $LP$-torsors). It is a stack fibered in groupoids over the category of $\BaseOfD$-schemes $S$ whose category $\scrH^1(\Spec \BaseOfD,L^+\BP)(S)$ consists of all $L^+\BP$-torsors (resp.\ $LP$-torsors) on $S$. The inclusion of sheaves $L^+\BP\subset LP$ gives rise to the natural 1-morphism 
			\begin{equation}\label{EqLoopTorsor}
				\scrH^1(\Spec \BaseOfD,L^+\BP)\longto \scrH^1(\Spec \BaseOfD,LP),~\CL_+\mapsto \CL\,.
			\end{equation}
		\end{definition}

		We let $\scrH^1(C,\FG)$ denote the category fibered in groupoids over the category of $\BF_q$-schemes, such that the objects over $S$, $\scrH^1(C,\FG)(S)$, are $\CG$-torsors over $C_S$ (also called $\CG$-bundles) and morphisms are isomorphisms of $\FG$-torsors. The resulting stack $\scrH^1(C,\FG)$ is a smooth Artin-stack locally of finite type over $\BF_q$; see \cite[Theorem~2.5]{AH_Global}.

		\section{$\FG$-shtukas and $\FG$-motives}\label{SectG-ShtAndGMotives}
		
		Recall that U. Jannsen proved that for a finite field $k$, the Grothendieck category of motives Mot(k), as defined via algebraic 
		correspondences modulo numerical equivalence, is a semi-simple 
		abelian category. Assuming the Grothendieck's Standard Conjecture D (stating that the numerical equivalence relation for algebraic cycles with rational coefficients coincides with the homological one), implies that there are fibre functors $\omega_\ell$, all $\ell= p,\infty$,
		and $\omega_{crys}$ such that
		$$\omega_\ell(M(X)) = H_{et}^\ast(X \otimes \BF, \BQ_\ell), \omega_{crys}(M(X)) = H_{crys}^\ast(X/W(\BF_q))\otimes B(\BF_q).
		$$
		Here $M(X)$ denotes the motive associated with a smooth projective schemes $X$ over $\BF_q$.
		
		Note that according to Deligne's philosophy, Shimura varieties with rational weight, may appear as moduli varieties for motives (with additional structures) \cite{De3}. This conception of Shimura varieties fits the case of Shimura varieties of \emph{abelian type}, and may fail in general, e.g. for Shimura varieties attached to exceptional groups. Nevertheless, as we will see in this section, one may achieve this for function field counterparts of Shimura varieties, i.e. for the moduli stack of global $\FG$-shtukas.  Indeed, one may interpret the moduli of $\FG$-shtukas as a moduli for motives in the following sense. We consider the category of $C$-motives $\mot(S)$ over $S$; see Definition \ref{DefCatCMotives}. This category is a slight modification of the category of Anderson $t$-motives. We proved that this category is semi-simple tannakian category, see \cite[Theorem 7.1]{CMot}. We define the category $\FG$-$\mot$ of $C$-motives with $\FG$-structure, in analogy with the category of motives with $G$-structure in the sense of \cite{Milne92}. We will see that this category admits crystalline (resp. \'etale) fiber functors at characteristic places (away from characteristic places).

		Let us fix a place $\nu\in C$ and a uniformizer $z:=z_\nu\in A_\nu=\hat{\CO}_{C,\nu}$. It yields canonical isomorphisms $A_\nu=\BF_\nu\dbl z\dbr$ and $Q_\nu=\BF_\nu\dpl z\dpr$. Let $L/\BF_\nu$ be a field extension. Let $\hat{\sigma}_\nu$ be the endomorphism of $L\dbl z\dbr=A_\nu\wh{\otimes}{_{\BF_\nu}} L$ which is the $\#\BF_\nu$-Frobenius $b\mapsto b^{\#{\BF_\nu}}=b^{q^{\deg\nu}}$ on $L$ and fixes $z$. Also let $L\dpl z\dpr=Q_\nu\wh{\otimes}{_{\BF_\nu}} L$ denote the fraction field of $L\dbl z\dbr$.

		\begin{definition}\label{Def_Crystal}
			Keep the above notation. 
			\begin{enumerate}
				\item
				A $\hat{\sigma}_\nu$-\emph{crystal} $\hat{\ul M}$ (resp. $\hat{\sigma}_\nu$-\emph{iso-crystal}) of rank $r$ over $L$ is a tuple $(\hat{M},\hat{\tau})$  (resp. $(\dot{\hat{M}},\hat{\tau})$ ) consisting of the following data
				
				\begin{enumerate}
					
					\item a free $L\dbl z\dbr$-modulue $\hat{M}$ (resp. an $L\dpl z\dpr$-vector space $\dot{\hat{M}}$) of rank $r$,
					
					\item an isomorphism $\hat{\tau}: \hat{\sigma}_\nu^\ast \hat{M}[1/z] \to \hat{M}[1/z]$ (resp. $\hat{\tau}: \hat{\sigma}_\nu^\ast \dot{\hat{M}} \to \dot{\hat{M}}$), where $\hat{M}[1/z]:=\hat{M}\otimes_{L\dbl z \dbr} L\dpl z \dpr$.

				\end{enumerate}
				A $\hat{\sigma}_\nu$-crystal is \emph{\'etale} if the isomorphism $\hat{\tau}: \hat{\sigma}_\nu^\ast \hat{M}[1/z] \to \hat{M}[1/z]$ comes from an isomorphism $\hat{\tau}: \hat{\sigma}_\nu^\ast \hat{M} \to \hat{M}$.\\ 
				
				\item A \emph{quasi-morphism} (resp. \emph{morphism}) between $\hat{\sigma}_\nu$-crystals $\hat{\ul M}:=(\hat{M},\hat{\tau})$ and $\hat{\ul M}':=(\hat{M}',\hat{\tau}')$ is a morphism $f:\hat{M}[1/z]\to \hat{M}'[1/z]$  (resp. $f:\hat{M}\to \hat{M}'$) such that $f\circ \hat{\tau}=\hat{\tau}' \circ \hat{\sigma}_\nu^\ast f$. We denote by $\hat{\sigma}_\nu$-$\textbf{CrystIso}(L)$ (resp. $\hat{\sigma}_\nu$-$\textbf{Cryst}(L)$) the $Q_\nu$-linear (resp. $A_\nu$-linear) category of $\hat{\sigma}_\nu$-crystals together with quasi-morphisms (resp. morphisms) as its morphisms. 

				\item
				A \emph{quasi-morphism} between $\hat{\sigma}_\nu$-iso-crystals $(\dot{\hat{M}},\hat{\tau})$ and $(\dot{\hat{M}}',\hat{\tau}')$ is a morphism $f:\dot{\hat{M}}\to \dot{\hat{M}}'$ such that $f\circ \hat{\tau}=\hat{\tau}' \circ \hat{\sigma}_\nu^\ast f$. We denote by $\hat{\sigma}_\nu$-$\textbf{IsoCryst}(L)$ the $Q_\nu$-linear category of $\hat{\sigma}_\nu$-iso-crystals together with quasi-morphisms as its morphisms.

				Note that we sometimes prefer not to specify any particular place $\nu\in C$, in this situation we omit the subscript $\nu$ from our notation.
				
			\end{enumerate}
			
		\end{definition}
		
		\noindent
		Below we recall the construction of the categories $\mot$ and $\FG$-$\mot$ from \cite{CMot}.

		\begin{definition}\label{DefCatCMotives}
			
			\begin{enumerate}
				\item
				Let $S$ be a scheme over $\BF_q$. A \emph{$C$-motive $\ul\CM$} with characteristic $\ul\nu=(\nu_1,\cdots,\nu_n)$ over $S$ is a tuple $(\CM,\tau_\CM)$ consisting of
				
				\begin{enumerate}
					\item \label{DefCatCMotives_i}
					a locally free sheaf $\CM$ of $\CO_{C_S}$-modules of finite rank,
					
					\item \label{DefCatCMotives_ii}
					an isomorphism $\tau_\CM: \sigma^\ast \dot{\CM} \to \dot{\CM}$ where $\dot{\CM}$ denotes the pullback of $\CM$ under the inclusion $\dot{C}_S\to C_S$, and $\sigma=id \times \sigma_S$ where $\sigma_S:S\to S$ is the absolute Frobenius morphism over $\BF_q$. Here $\dot{C}:=C\setminus\{\nu_1,\cdots,\nu_n\}$.   
				\end{enumerate}
				
				\item 
				A \emph{morphism} $\ul\CM\to\ul\CN$ is a homomorphism $f:\CM\to\CN$ of sheaves on $C_S$ such that $f\circ\tau_\CM|_{\dot{C}_S}=\tau_\CN|_{\dot{C}_S}\circ \sigma^\ast f$. The set of morphisms is denoted $\Hom_S(\ul\CM,\ul\CN)$. The set of \emph{quasi-morphisms} $\QHom (\ul\CM,\ul\CN)$ consists of the equivalence classes of the commutative diagrams 
				$$
				\CD
				\sigma^\ast \dot{\CM} @>\tau_\CM>> \dot{\CM}\\
				@V{\sigma^\ast f}VV @VVfV\\
				\sigma^\ast \dot{\CN}\otimes\CO(D_S)@>\tau_\CN>> \dot{\CN}\otimes \CO(D_S),
				\endCD  
				$$
				
				where $D$ is a divisor on $C$ and $D_S:=D\times_{\BF_q} S$, and two such diagrams for divisors $D$ and $D'$ are called equivalent provided that the corresponding diagrams agree when we tensor with $\CO(D_S+D_S')$. When $S=\Spec L$, for a field $L$, one can equivalently say that the set of quasi-morphisms $\QHom (\ul\CM,\ul\CN)$ is given by the following commutative diagrams 
				
				$$
				\CD
				\sigma^\ast \CM_\eta @>\tau_{\CM,\eta}>> \CM_\eta\\
				@V{\sigma^\ast f}VV @VVfV\\
				\sigma^\ast \CN_\eta@>\tau_{\CN,\eta}>> \CN_\eta.
				\endCD  
				$$
				
				Here $\CM_\eta$ denotes the pull back of $\CM$ under $\eta\times_{\BF_q} S \to C_S$.

				\item
				
				A \emph{quasi-isogeny} between $\ul \CM$ and $\ul \CN$ is a morphism in $\QHom (\ul\CM,\ul\CN)$ which admits an inverse. 
				
				\item
				We denote by $\mot (S)$ the $Q$-linear category whose objects are $C$-motives of characteristic $\ul\nu$ as above, with quasi-morphisms as its morphisms. We further denote by $\mot (S)^\circ$\forget{ (resp. $\mot (L)^\circ$)} the category obtained by restricting the set of morphisms to quasi-isogenies\forget{ (resp. isogenies)}. When $S=\Spec L$ we simply use the notation $\mot (L)$.
			\end{enumerate}

		\end{definition}

		One may endow the category of $C$-motives $\mot(S)$ with a $G$-structure. This leads to the following definition.
		
		\begin{definition}\label{DefGMotives}
			Let $\Rep \FG$ denote the category of representations of $\FG$ in finite free $\cO_C$-modules $\CV$.  By a \emph{$\FG$-motive} (resp. \emph{$G$-motive}) over $S$ we mean a tensor functor $\ul \CM_\FG: \Rep \FG \to \mot(S)$ (resp. $\ul \CM_G: \Rep_Q G \to \mot(S)$). We say that two $\FG$-motives (resp. $G$-motives) are \emph{isomorphic} if they are isomorphic as tensor functors. We denote the resulting category of $\FG$-motives (resp. $G$-motives) over $S$ by $\FG$-$\mot(S)$ (resp. $G$-$\mot(S)$).
		\end{definition}
		
		\noindent
		Note that the construction of the category
		$G$-$\mot(S)$ (resp. $\FG$-$\mot(S)$) is functorial in $G$ (resp. -$\FG$).
		
		\noindent
		Let us now recall the following definition of the moduli of global $\FG$-shtukas. 
		
		\begin{definition}\label{DefGlobal_Sht}
			\begin{enumerate}
				\item\label{DefGlobal_Sht_Obj}
				A \emph{global $\FG$-shtuka} $\ul\CG$ over an $\BF_q$-scheme $S$ is a tuple $(\CG,\ul\charsect,\tauGlob)$ consisting of 
				\begin{enumerate}
					\item[-]
					a $\FG$-bundle $\CG$ over $C_S$, 
					\item[-]
					an $n$-tuple $\ul\charsect$ of (characteristic) sections and 
					\item[-]
					an isomorphism $\tauGlob\colon  \s \CG|_{C_S\setminus \Gamma_{\ul\charsect}}\isoto \CG|_{C_S\setminus  \Gamma_{\ul\charsect}}$.

				\end{enumerate}
				
				We let $\nabla_n\scrH^1(C,\mathfrak{G})$ denote the stack whose $S$-points parameterizes global $\FG$-shtukas over $S$. This is an ind-Deligne-Mumford stack which is ind-separated and of ind-finite type, according to \cite[Theorem 3.15]{AH_Global}. Sometimes we will fix the sections $\ul s:=(\charsect_i)_i\in C^n(S)$ and simply call $\ul\CG=(\CG,\tauGlob)$ a global $\FG$-shtuka over $S$.

				\forget
				{
					--------------------------------
					\item\label{RemNablaHisIndDM}
					
					It can be shown that the stack $\nabla_n \scrH^1(C,\FG)$ is an ind-Deligne-Mumford stack over $C^n$ which is ind-separated and locally of ind-finite type. Indeed, a choice of faithful representation of $\FG$ in $\SL_r$ induces an ind-structure 
					$$
					\nabla_n \scrH^1(C,\FG):=\dirlim[\ul\omega]\nabla_n^\ul\omega \scrH^1(C,\FG)
					$$ 
					such that $\nabla_n^\ul\omega \scrH^1(C,\FG)$ are Deligne-Mumford. Here $\ul\omega$ runs over n-tuple of coweights of $\SL_r$. For the proof see \cite[Theorem 3.15]{AH_Global}. 
					----------------------------------
				}
				
				\item \label{DefGlobal_Sht_Mor}
				We let $\nabla_n\scrH^1(C,\FG)(S)_Q$ denote the category which has the same objects as $\nabla_n\scrH^1(C,\FG)(S)$, but the set of morphisms is enlarged to \emph{quasi-isogenies} of $\FG$-shtukas. A quasi-isogeny  $f:\ul\CG\to\ul\CG'$ is a commutative diagram
				
				$$
				\CD
				\CG@>f>>\CG'\\
				@A{\tau}AA @AA{\tau'}A\\
				\sigma^\ast\CG @>\sigma^\ast f >>\sigma^\ast\CG',
				\endCD
				$$
				defined outside $\Gamma_\ul s \cup D\times_{\BF_q}S$ for a closed subscheme $D\subseteq C$. Note that as a morphism of torsors $f$ is automatically an isomorphism. We denote by $QIsog_S(\ul\CG,\ul\CG')$ the set of quasi-isogenies between $\ul\CG$ and $\ul\CG'$.

				\item\label{DefGlobal_Sht_Completion}
				We denote by $\nabla_n\scrH^1(C,\FG)^{\ul\nu}$, the formal stack which is obtained by taking the formal completion of the stack $\nabla_n\scrH^1(C,\FG)$ at a fixed $n$-tuple of pairwise different characteristic places $\ul\nu=(\nu_1,\ldots,\nu_n)$ of $C$ in the sense of  \cite[Definition A.12]{Har1}. This means we let $A_{\ul\nu}$ be the completion of the local ring $\CO_{C^n,\ul\nu}$, and we consider global $\FG$-shtukas only over schemes $S$ whose characteristic morphism $S\to C^n$ factors through $\Nilp_{A_\ul\nu}$. We similarly denote by $\nabla_n\scrH^1(C,\FG)^{\ul\nu}(S)_Q$ the category over $S$, with morphisms enlarged to quasi-isogenies.

			\end{enumerate}
		\end{definition}

		In contrast with global situation one defines

		\begin{definition}\label{DefP-crystal}
			The category $\BP$-$\hat{\sigma}$-$\textbf{Cryst}_\BF(L)$ of $\BP$-$\hat{\sigma}$-crystals (or $\hat{\sigma}$-crystals with $\BP$-structure) over $L$ is the category whose objects are the tensor functors
			
			$$
			\Rep \BP \to \hat{\sigma}-\textbf{Cryst}_\BF(L),
			$$
			with natural transformations of functors as the set of morphisms. Similarlly we define the category of $P$-$\hat{\sigma}$-$\textbf{IsoCryst}_\BF(L)$ as the category whose objects are tensor functors from the representation category $\Rep P$ to the category $\hat{\sigma}$-$\textbf{IsoCryst}(L)$, with obvious set of morphisms.
			
		\end{definition}

		\begin{definition}\label{localSht}
			\begin{enumerate}
				\item
				A \emph{local $\BP$-shtuka} (resp. \emph{local $P$-isoshtuka} ) over $S\in \Nilp_{\BaseOfD\dbl\zeta\dbr}$ is a pair $\ul \CL = (\CL_+,\tau)$ (resp. $\ul \CL = (\CL,\tau)$) consisting of an $L^+\BP$-torsor (resp. $L\BP$-torsor) $\CL_+$ (resp. $\CL$) on $S$ and an isomorphism of the (associated) loop group torsors $\tauLoc\colon  \hat{\sigma}^\ast \CL \to\CL$. 
				
				\item
				A \emph{quasi-isogeny} $f\colon\ul\CL\to\ul\CL'$ between two local $\BP$-shtukas $\ul{\CL}:=(\CL_+,\tau)$ and $\ul{\CL}':=(\CL_+' ,\tau')$ over $S$ is an isomorphism of the associated $LP$-torsors $f \colon  \CL \to \CL'$ such that the following diagram  
				
				\[
				\xymatrix {
					\sigma^\ast\CL \ar[r]^{\tau} \ar[d]_{\sigma^\ast f} & \CL\ar[d]^f  \\
					\sigma^\ast\CL' \ar[r]^{\tau'}&  \CL' \;.
				}
				\]

				becomes commutative. For local $P$-isoshtukas, this can be defined similarly. 
				\item
				We denote by $\QIsog_S(\ul{\CL},\ul{\CL}')$ the set of quasi-isogenies between local $\BP$-shtukas (resp. $P$-isoshtukas) $\ul{\CL}$ and $\ul{\CL}'$ over $S$. We denote by $Sht_\BP(S)$ (resp. Iso--$Sht_P(S)$) the category of local $\BP$-shtukas (resp. $P$-isoshtukas) over $S$ with quasi-isogenies.

			\end{enumerate}
		\end{definition}

		\begin{remark}\label{RemFunctPSHtToPisoSHt}
			
			a) There is a fully faithful functor from category of local $\BP$-shtukas to the category of local $P$-isoshtukas, induced by the obvious morphism $$\breve{H}^1(S,L^+\BP)\to \breve{H}^1(S,LP)$$ corresponding to the inclusion $L^+\BP\into LP$. 
			
			b) Note that by tannakian formalism there is an equivalence of categories between the category of local $P$-isoshtukas over $L$ and the category $P$-$\hat{\sigma}$-$\textbf{IsoCryst}_\BF(L)$.
			Consequently there is the following diagram of functors
			
			\[
			\xymatrix {
				Sht_\BP(L) \ar[r] \ar[d] & \text{Iso-}Sht_P(S)\ar[d]^\cong  \\
				\BP\text{-}\hat{\sigma}\text{-}\textbf{Cryst}_\BF(L) \ar[r] &  P\text{-}\hat{\sigma}\text{-}\textbf{IsoCryst}_\BF(L) \;.
			}
			\]
			
		\end{remark}
		
		\forget{
			\begin{remark}\label{Rem_Rig_Loc}
				Like global $\FG$-shtukas, quasi-isogenies of local $\BP$-shtukas have the rigidity property. This means that they lift over infinitesimal thickenings; see \cite[Proposition 2.11]{AH_Local}.
			\end{remark}
		}

		\begin{lemma}\label{LemPSHtLocTrivForEtTop}
			Let $\BP$ (resp. $\FG$) be a smooth affine group scheme of finite type over $\BD$ (resp. $C$) with generic fiber $P$ (resp. $G$). Then 
			
			\begin{enumerate}
				
				\item 
				any local $\BP$-shtuka $\ul\CL=(\CL_+,\tau)$ over $S=\Spec R$, is \'etale locally on $S$ isomorphic to $(L^+\BP,b\hat{\sigma})_{S'}$ for an \'etale cover $S'\to S$.
				
				\item 
				any local $P$-isoshtuka $\ul\CL=(\CL,\tau)$ over $S$ with constant isogeny class $[b]$ is \'etale locally on $S$ isomorphic to  $(LP,b\hat{\sigma})_{S'}$ for an \'etale cover $S'\to S$.
				
				\item 
				Assume that $G$ is connected reductive, then any $\FG$-shtuka $\ul\CG$ over $\BF_q^\alg$ can be trivialized up to a quasi-isogeny.

			\end{enumerate}
			
		\end{lemma}

		\begin{proof}

			\begin{enumerate}
				
				\item The category of formal $\hat{\BP}$-torsors on $\hat{\BD}_R:=\Spf R\dbl z\dbr$, see \cite[Definition 2.2]{AH_Local}, with $\hat{\BP}$ the completion of $\BP$ along $V(z)$, is equivalent to the category of $L^+\BP$-torsors on $S$; see \cite[Proposition 2.4]{AH_Local}. Let $\wh\CP$ be a formal $\hat{\BP}$-torsor corresponding to $\CL^+$. Since $\BP$ is smooth, the torsor $\wh\CP\times_{\hat{D}_R} R\dbl z\dbr/z$ admits a section over some \'etale cover $S':=\Spec R'\to S$. Again by smoothness of $\BP$ the section lifts to a section of $\wh\CP\times_{\hat{\BD}_R} R'\dbl z\dbr/z^n$. This eventually gives a trivialization of $\wh\CP$, and therefore $\CL_+$, which induces a trivialization of $\ul\CL$. 
				
				\item
				This is \cite{H-KII} theorem 2.11.

				\item 
				Let $\BF_q^\alg(C)$ be the function field of a curve
				$C_{\BF_q^{alg}}$, defined over an algebraically closed field $\BF_q^{alg}$, for connected reductive group $G$ over $\BF_q^\alg(C)$, we have $H^1(\BF_q^\alg(C), G) = 0$. 
				Recall that this is the celebrated Steinberg’s Theorem \cite{Stein}, that was extended by Borel-Springer \cite{BS} to the
				case of arbitrary fields after restricting to connected reductive groups. In particular one observes that a $\FG$-shtuka $\ul\CG$ in $\nabla_n\scrH^1(C,\FG)(\BF_q^\alg)$ can be trivialized up to a quasi-isogeny.

			\end{enumerate}
			
		\end{proof}

		\begin{theorem}\label{ThmGroupTheoreticDescriptionOfIndexSet}
			
			Consider the category $\CM ot_C^\ul\nu(L)$ of $C$-motives over $L$, for $\ul\nu \subset C$ over ${\BF}_q^{\alg}$. We have the following statements
			
			\begin{enumerate}
				\item\label{ItemCMotisTannkianCat}
				It is non-neutral Tannakian category, with a fibre functor
				
				\begin{equation}\label{EqFibFunct}
					\omega: \CM ot_C^\ul\nu(L) \longrightarrow \textrm{Vec}_{Q_L}.
				\end{equation}
				Here $Q_L$ denote the function field of $C_L$.
				\item\label{ItemRealizationFunctorsBase}
				$\CM ot_C^\ul\nu(L)$ admits \'etale (resp. Crystalline ) realization functor 
				\begin{equation}\label{EqEtRealization}
					\omega^\nu(-):\mot(L)\to Q_\nu[\Gamma_L]\textbf{-modules}
				\end{equation}
				(resp. 
				
				\begin{equation}\label{EqCrystRealization}
					\omega_\nu(-):\mot(L)\longto\textbf{$\hat\sigma_\nu$-IsoCryst}(L)
				\end{equation}
				)
				at $\nu\in C\setminus \ul\nu$ (resp. $\nu\in\ul\nu$). Moreover the \'etale realization functor $\omega^\nu(-)$ is exact.

				\item\label{ItemTateConj}
				(Analog of Tate conjecture) Let $\ul\CM$ and $\ul\CM'$ be $C$-motives over a finite field $L$. Let $\nu$ be a closed point of $C$ away from the characteristic places $\nu_i$. Then there is an isomorphism
				
				$$
				\QHom_L(\ul\CM,\ul\CM')\otimes_Q Q_\nu \tilde{\longrightarrow} \Hom_{{Q_\nu}[\Gamma_L]}(\omega^\nu(\ul\CM),\omega^\nu(\ul\CM')). 
				$$
				Moreover if $\ul\CM=\ul\CM'$ then the above are ring isomorphisms.

				\item\label{ItemJannsenssemisimplicity} (Analog of Jannsen's semi-simplicity) The category $\mot:=\mot(\BF_q^\alg)$ is semi-simple.
				\item\label{ItemAnalogofHonda-Tate theory} (Analog of Honda-Tate theory)
				Set $W_\ul\nu=\{\alpha\in Q^\alg ; \nu(\alpha)= 0 ~\forall~ \nu\notin \ul\nu  \}$. There is a bijection 
				
				$$
				\text{set $\Sigma$ of simple objects in $\mot$} \leftrightarrow \text{elements of $\Gamma_Q\backslash W_\ul\nu\times \BN_{\geqslant 1}\slash\sim$}.
				$$
				
				Here $(\alpha,n)$ and $(\beta,m)$ are equivalent if $\alpha^{m.l}=\beta^{n.l}$  for some integer $l\in \BN_{\geqslant 1}$.
				
				
				\item\label{ItemExtension1}
				The Tannakian fundamental groupoid $\FP$ over $\Breve{Q} / Q$ of $\CM ot_C^{\ul\nu}$ is an extension
				\[
				1 \longrightarrow W_\ul\nu \longrightarrow \FP \longrightarrow \Gal (\Breve{Q} / Q) \longrightarrow 1\,.
				\]
				The kernel group $P:=\FP^{\Delta}$\forget{of the corresponding motivic groupoid $\FP:=\Aut^\otimes\bigl(\omega|\mot(\ol\BF_q)\bigr)$} is a pro-reductive group over $\Breve{Q}$.

				\forget{
					
					\item\label{ItemExtension2}
					The Tannakian fundamental groupoid $\FP_\nu$ over $\Breve{Q}_\nu / Q_\nu$ of $\hat{\sigma}$-$\textbf{Cryst}_{\BF_{\nu}}(\BF_q^{alg})$ is an extension
					\[
					1 \longrightarrow ??  \longrightarrow \FP_\nu \longrightarrow \Gal(\breve{Q}_\nu/ Q_\nu) \longrightarrow 1\,.
					\]
					The kernel group $P_\nu:=\FP_\nu^{\Delta}$\forget{of the corresponding motivic groupoid $\FP:=\Aut^\otimes\bigl(\omega\bigr)$} is a pro-reductive group over $\breve{Q}_{\nu}$.
				}

			\end{enumerate}

			
		\end{theorem}

		\begin{proof} \ref{ItemCMotisTannkianCat} One can see easily that the category is tannakian, see \cite[Proposition 2.4]{CMot}. It admits the following fibre functor over $\breve{Q}$
			
			\begin{equation*}
				\omega: \CM ot_C^\ul\nu \longrightarrow \textrm{Vec}_{\,{\BF}_q^{\alg} (C)},\quad \underline{\CG} \longmapsto \CG_{\ol\eta}.
			\end{equation*}

			where $\CG_\ol\eta := \CG \otimes_{{\cal O}_{C_{\BF_q^{\alg}}}} \breve{Q}$ is the generic fiber of $\CG$ on $C_{\BF_q^{\alg}}$.
			
			\ref{ItemRealizationFunctorsBase} For the construction of \'etale realization functor $\omega^\nu$ and crystalline realization functor $\omega_\nu$; see \cite[Section 2.1.2 and Section 2.1.3]{CMot}. For the exactness of the functor see  \cite[Proposition 2.1.4]{CMot}.
			
			\ref{ItemTateConj} This is proved in \cite{CMot}[Theorem 2.16].

			\ref{ItemJannsenssemisimplicity}  The idea is similar to the Jannsen's argument, based on studying the center of endomorphism algebra associated to a $C$-motive. We proved this semisimplicity result in \cite{CMot}[Theorem~7.1].

			\ref{ItemAnalogofHonda-Tate theory} We recall how one assigns an element of $\Gamma_Q\backslash W_\ul\nu\times \BN_{\geqslant 1}\slash\sim$
			to a $C$-motive. First notice that the category $\mot$ is direct limit of $\mot(E/\BF_q)$ for finite field extension $E/\BF_q$, see \cite{CMot}[Proposition 4.5].

			Let $\ul\CM:=(\CM,\tau_\CM)$ be a simple object in $\mot(\ol\BF_q)$. Suppose that it comes by base change from a $C$-motive in $\mot(E)$ for a finite extension $E/\BF_q$ of degree $n$. Let $\pi:=\pi_\ul\CM$ denote the corresponding Frobenius isogeny. Let $\alpha_\pi$ be a zero of the minimal polynomial $\mu_\pi$.  Then sending $\ul\CM$ to the pair $(\alpha,n)$ gives the assignment $\Sigma\to \Gamma_Q\backslash W_\ul\nu\times \BN_{\geqslant 1}\slash\sim$. To see that this assignment is one to one (resp. onto) see \cite{CMot}[Theorem 3.5] (resp. \cite{FelixThesis}[Theorem 3.12]).

			\ref{ItemExtension1} The first part follows from \cite[Theorem 1.5]{CMot}. The fact that the kernel group $P:=\FP^{\Delta}$ of the corresponding motivic groupoid $\FP:=\Aut^\otimes\bigl(\omega|\mot(\ol\BF_q)\bigr)$ is a pro-reductive group over $\breve{Q}$ now follows from \cite[Proposition 2.23]{De-Mi}.\\

		\end{proof}

		\begin{proposition}\label{PropEquivOfMotCats}
			Keep the notation in section \ref{SubsectNotationandConventions}. We have the following statements
			
			\begin{enumerate} 
				
				\item \label{ItemG-ShtToG-Mot}
				
				There is a functor

				\begin{equation}\label{Eq_FunctShtToGMot}
					\ul\CM_\FG(-): \nabla_n\scrH^1(C,\FG)^\ul\nu(L)\;\longto\; \FG\text{-}\mot(L), \qquad
					\ul\CG\mapsto \ul\CM_\FG(\ul\CG),
				\end{equation}
				
				see Definition \ref{DefGMotives} for the definition of the category of $\FG$-motives $\FG\text{-}\mot(L)$, which induces a fully faithful functor 
				
				\begin{equation}\label{Eq_FunctShtToGMotI}
					\ul\CM(-): \nabla_n\scrH^1(C,\FG)^\ul\nu(L)_Q\;\longto\; G\text{-}\mot(L), \qquad
					\ul\CG\mapsto \ul\CM(\ul\CG),
				\end{equation}
				
				\item\label{ItemRealizationA} The category  $\nabla_n\scrH^1(C,\FG)^{\ul\nu}(S)$ of $\FG$-shtukas over $S$ in $\Nilp_{A_{\ul\nu}}$ admits an \'etale realization functor
				
				\begin{eqnarray}\label{tatefunctor}
					\omega^\ul\nu(-)\colon \nabla_n\scrH^1(C,\FG)^{\ul\nu}(S) \;&\longto &\; Funct^\otimes (\Rep_{\BO^{\ul\nu}}\FG\,,\,\FM od_{\BA^{\ul\nu}[\pi_1^\et(S,\bar{s})]})\,\label{rationaltatefunctor}
				\end{eqnarray}
				
				For $S=\Spec L$ the above functor is induced by the obvious \'etale realization functor $\omega^\ul\nu(-)$ on $\FG\text{-}\mot(L)$ which is compatible with the one induced by (\ref{EqEtRealization}).
				
				\item\label{ItemRealizationB}
				
				There is a crystalline realization functor 
				
				\begin{equation}\label{Eq_CrystRealization}
					\omega_{\nu_i}(-)\colon \nabla_n\scrH^1(C,\FG)^{\ul\nu}(S)\to \Sht_{\BP_{\nu_i}}^{\Spec A_{\nu_i}}(S)
				\end{equation}
				
				which assigns to a global $\FG$-shtuka $\ul\CG$ over $S$ its local $\BP_{\nu_i}$-shtuka $\omega_{\nu_i}(\ul\CG)$ at $\nu_i$. Furthermore, for a field $L/\BF_{\nu_i}$ there is an equivalence of categories between $\Sht_{\BP_{\nu_i}}^{\Spec A_{\nu_i}}(L)$ and $\BP_{\nu_i}$-$\hat{\sigma}$-$\textbf{Cryst}_{\BF_{\nu_i}}(L)$. Consequently, the above functor is induced by the realization functor 
				
				\begin{equation}\label{Eq_FunctG-MotToP-Cryst}
					\FG\text{-}\mot(L)\to\BP_{\nu_i}\text{-}\hat{\sigma}\text{-}\textbf{Cryst}_{\BF_{\nu_i}}(L)
				\end{equation}
				induced by (\ref{EqCrystRealization}).
				
			\end{enumerate}

		\end{proposition}

		\begin{proof}
			\ref{ItemG-ShtToG-Mot}
			We recall the construction of the functor (\ref{Eq_FunctShtToGMot}). For this we recall that the assignment $\FG\mapsto\nabla_n\scrH^1(C,\FG)$ is functorial. Namely, a morphism $\iota:\FG\to\FG'$ gives rise to a morphism 
			
			$$
			\iota_\ast:\nabla_n\scrH^1(C,\FG)\to\nabla_n\scrH^1(C,\FG'), 
			$$
			which is induced by the push out morphism $\scrH^1(C,\FG)\to\scrH^1(C,\FG')$ defined by sending $\CG\to \CG\times_{\FG,\iota}\FG'$. Assuming that $\FG'=\GL(\CV)$ for a vector bundle $\CV$ over $C$, by functoriality we obtain $\iota_\ast: \scrH^1(C,\FG)\to\scrH^1(C,\GL(\CV))$. Using the equivalence of categories $\scrH^1(C,\FG')(S)=Vect_C(S)$, Where $Vect_C$ is the stack of vector bundles over $C$, we can view $\iota_\ast\ul\CG$ as a $C$-motive. Consequently, we obtain a functor
			
			$$
			\nabla_n\scrH^1(C,\FG)\to \FG\text{-}\mot(L).
			$$
			Let us now explain why the induced functor 
			$$
			\nabla_n\scrH^1(C,\FG)_Q\to G\text{-}\mot(L)
			$$
			is fully faithful. Let $\ul\CG$ and $\ul\CG'$ be two global $\FG$-shtukas in $\nabla_n\scrH^1(C,\FG)_Q$, and let $\phi:\CM(\ul\CG)\to\CM(\ul\CG')$ be a morphism between the corresponding objects in $G$-$\mot(L)$. The morphism $\phi$ induces a natural transformation 
			$$
			\omega(\phi): \omega\left(\CM(\ul\CG)(-)\right)\to \omega\left(\CM(\ul\CG')(-)\right)
			$$
			between the associated functors from $\Rep G$ to the category of $Q_L$-vector spaces. Set $F_{\ul\CG}:=\omega\left(\CM(\ul\CG)(-)\right)$. The Frobenius isogeny $\tau:\sigma^\ast\ul\CG\to\ul\CG$ (resp. $\tau':\sigma^\ast\ul\CG'\to\ul\CG'$) induces a morphism $F_{\tau}:F_{\sigma^\ast\ul\CG}\to F_{\ul\CG}$ (resp. $F_{\tau'}: F_{\sigma^\ast\ul\CG'}\to F_{\ul\CG'}$), with $\omega(\phi)\circ F_\tau=F_{\tau'}\circ \omega(\sigma^\ast (\phi))$. By Tannakian formalism we get a morphism $\phi_\eta:\CG_\eta\to\cG_\eta'$ of $G$-torsors over $Q_L$, such that $\phi_\eta\circ\tau_\eta=\tau_\eta'\circ\sigma^\ast\phi_\eta$. This extends to a quasi-isogeny $\ul\CG\to\ul\CG'$. 
			
			\ref{ItemRealizationA}
			For the construction of this functor see \cite[Chapter 6]{AH_Global}. There we used the notation $\check{\CV}_{-}$.
			
			\ref{ItemRealizationB} For the construction of the crystalline realization functor see \cite[Chapter 5]{AH_Local}. There we use the notation $\wh{\Gamma}_{\nu_i}(-)$ and it was called the global-local functor. To see the equivalence of the categories one proceeds in the following way.\\
			Let $B \subseteq GL_r$ be the Borel subgroup of upper triangular matrices and let $T$ be the torus of diagonal matrices. Then $X_\ast(T)= \BZ^r$ with simple coroots $e_i -e_{i+1}$ for $i = 1,...,r-1$. Also $X^\ast(T) = \BZ^r$. Let $\lambda_i = (1,...,1,0,... ,0)$ with multiplicities $i$ and $r-i$. The Weyl module $V (\lambda_1) =Ind_B^{GL_r} (-\lambda_1)_{dom}$ of highest weight $\lambda_1$ is simply the standard representation of $GL_r$ on the space of column vectors with $r$ rows, and $V (\lambda_i) = \wedge_i V (\lambda_1)$. For an $\BF_q$-scheme $S$ we have $L^+\GL_r(S) = \GL_r \Gamma(S,\CO_S)\dbl z\dbr$. There is an equivalence between the category of $L^+\GL_r$-torsors on $S$ and the category of sheaves of $\CO_S\dbl z\dbr$-modules which Zariski-locally on $S$ are free of rank r with isomorphisms as the only morphisms; see \cite[Lemma~4.2]{H-V}. According to this equivalence, one sends $\CL$ to the sheaf $\CL_{\lambda_1}$ corresponding to the following presheaf 
			$$
			Y \mapsto \left(\CL(Y ) \times \left(V(\lambda_1) \otimes_{\BF_q} \CO_S\dbl z\dbr(Y )\right)\right)/\GL_r(Y \dbl z \dbr); 
			$$
			Accordingly, the category of local $\GL_r$-shtukas over $\Spec L$ with quasi-isogenies as morphisms is equivalent to the subcategory of rank $r$ $\hat{\sigma}$-crystals in $\textbf{$\hat{\sigma}$-Cryst}(L)$. 

			A morphism $\rho: \BP\to \BP'$ induces a morphism $L^+\BP\to L^+\BP'$. This gives the functor
			$$
			\rho_\ast: \scrH^1(\Spec \BF, L^+\BP)\to\scrH^1(\Spec \BF, L^+\BP).
			$$
			The latter further induces
			$$
			\rho_\ast: Sht_\BP(S)\to Sht_{\BP'}(S),~\ul\CL\mapsto \rho_\ast \ul\CL.
			$$
			In particular from the above explanation, we can define the following pairing
			\begin{equation}\label{EqPshtandPMot}
				Sht_\BP(L) \times \Rep_{\BF\dbl z \dbr} \BP \to \textbf{$\hat{\sigma}$-Cryst}(L),~\ul\CL\times \rho\mapsto \rho_\ast\ul\CL.
			\end{equation}
			In other words a local $\BP$-shtuka $\ul\CL$  gives rise to a $\BP$-$\hat{\sigma}$-crystal over $L$. Conversely, given a functor
			
			\begin{equation}
				\Rep_{\BF\dbl z \dbr} \BP \to \textbf{$\hat{\sigma}$-Cryst},~ \rho\mapsto (\hat{M}_\rho,\hat{\tau}_\rho:\sigma^\ast \hat{M}_\rho[1/z]\to \hat{M}_\rho[1/z])
			\end{equation}
			we proceed in the following way. Namely, consider the functor $\Rep_{\BF\dbl z \dbr} \BP \to L\dbl z\dbr/z^n\text{-modules}$, defined by sending $\rho$ to $\hat{M}_\rho\otimes L\dbl z \dbr/z^n$. By tannakian formalism this gives a $\BP_n$-torsor $\hat{\CP}_n$, over $\BD_{n,L}:=\Spec L\dbl z\dbr/z^n$, where $\BP_n:=\BP\times_\BD \BF\dbl z\dbr/z^n$. Therefore we obtain a formal $\hat{\BP}$-torsor $\hat{\CP}$ over $\hat{\BD}_L$, which by the equivalence of the categories \cite[Proposition 2.4]{AH_Local} gives a $L^+\BP$-torsor $\CL_+$ over $\Spec L$. Let $\CL$ denote the corresponding $LP$-torsor. We may view $\CL$ as the functor 
			$$\Rep P\to L\dpl z \dpr\text{-vector spaces},$$
			given by $\rho\mapsto \hat{M}_\rho\otimes_L L\dpl z \dpr$, and the Frobenius morphisms $\hat{\tau}_\rho$ as a natural transformation of the functors $\hat{\sigma}^\ast\CL\to \CL$. This gives the local $\BP$-shtuka $\ul\CL$. 
			
		\end{proof}

		\begin{remark}
			Note that the functor \ref{Eq_FunctShtToGMot} is not essential surjective. See the remark \ref{RemGeometricisGroupTheoretic} below.
		\end{remark}

		\begin{remark}\label{RemarkRealizationOfGMotives}
			There is a realization functor
			$$
			\omega^\ul\nu(-): \FG\text{-}\mot(L)\to Funct^\otimes (\Rep_{\BO^{\ul\nu}}\FG\,,\,\FM od_{\BA^{\ul\nu}[\Gamma_L]}),
			$$
			
			induced by \ref{EqEtRealization}. Note that by construction, for a $\FG$-shtuka $\ul\CG$, the associated functors $\omega^\ul\nu(\ul\CG)$, and $\omega^\ul\nu(\ul\CM_\FG(\ul\CG))$, see \ref{Eq_FunctShtToGMot}, coincide.
			
		\end{remark}
		
		The above proposition \ref{PropEquivOfMotCats} has the following obvious corollary.
		
		\begin{corollary}\label{CorQIsogEqAut}
			There is an isomorphism $QIsog(\ul\CG)\cong \Aut(\ul\CM(\ul\CG))$.
		\end{corollary}
		
		\begin{remark}\label{RemarkEquivOFMotCats}
			\begin{enumerate}
				\item\label{ItemDualTannakianCat1}
				Consider the category $\FG\text{-}\breve{\CM ot} (\BF_q^{\alg})$ whose objects are morphisms of groupoids 
				\begin{equation*}
					\varphi : \FP \longrightarrow \FG \ \times_C \ \Spec\bigl( \breve{Q} \otimes_Q \breve{Q}\bigr)
				\end{equation*}
				with obvious set of morphisms. Regarding the fiber functor \ref{EqFibFunct}, the functor \ref{Eq_FunctShtToGMot} induces a fully faithful functor

				\begin{equation}\label{EqGShtAndGerbs}
					\varphi_{-}: \nabla_n\scrH^1(C,\FG)^\ul\nu(\BF_q^{\alg})_Q\to \FG\text{-}\breve{\CM ot} (\BF_q^{\alg}).
				\end{equation}

				which assigns $\varphi_{\ul{\CG}}$ to a $\FG$-shtuka $\ul\CG$. Here $\nabla_n\scrH^1(C,\FG)^\ul\nu(\BF_q^{\alg})_Q$ is the category   $\nabla_n\scrH^1(C,\FG)^\ul\nu(\BF_q^{\alg})$ with set of morphisms extended to quasi-isogenies of $\FG$-shtukas.  In particular $QIsog(\ul\CG)\cong \Aut (\varphi_{\ul\CG})$. The realization functors in \ref{ItemRealizationA} and \ref{ItemRealizationB} induce the corresponding realization functors $\omega^\ul\nu(-)$ and $\omega_\nu(-)$ on $\FG-\breve{\CM ot} (\BF_q^{\alg})$.

				\item\label{ItemDualTannakianCat2}
				Consider the category $\BP_\nu\text{-}\hat{\sigma}\text{-}\breve{\textbf{Cryst}}_{\BF_{\nu}}(\BF_q^{alg})$ whose objects are morphisms of groupoids
				$$
				\FP_\nu\to \BP_\nu \times_{A_\nu} (\Spec \breve{Q}_\nu\times_{Q_\nu}\Spec \breve{Q}_\nu).
				$$
				
				Here $\FP_\nu$ is the Tannakian fundamental groupoid of the category $\hat{\sigma}$-$\textbf{Cryst}_{\BF_{\nu}}(\BF_q^{alg})$. When we do not want to specify the place $\nu\in C$, we omit the subscript $\nu$ from our notation. Note that this category is equivalent to the category of local $\BP$-shtukas over $\BF_q^\alg$. 
				
				\item\label{ItemCommuteDiageram}
				
				There is the following diagram 
				$$
				\CD
				\nabla\scrH^1(C,\FG)(L)@>{\ul\CM_\FG^{-}}>>\FG\text{-}\mot(L)\\
				@V{\omega_\nu(-)}VV @VVV\\
				Sht_{\BP_\nu}^{\Spec A_\nu}(L)@>>> \BP_\nu\text{- $\hat{\sigma}$-}\textbf{Cryst}_{\BF_{\nu}}(L)
				\endCD
				$$
				concerning the global-local compatibility of the functors
			\end{enumerate}
			
		\end{remark}

		\section{Langlands-Rapoport Conjecture and Special Points}\label{Sect_Proof_Of_LR_Conj}

		\subsection{Boundedness Condition}

		Fix an algebraic closure $\BaseOfD\dpl\zeta\dpr^\alg$ of $\BaseOfD\dpl\zeta\dpr$. For a finite extensions of discrete valuation rings $R/\BaseOfD\dbl\zeta\dbr$ with $R\subset\BaseOfD\dpl\zeta\dpr^\alg$, we denote by $\kappa_R$ its residue field, and we let $\Nilp_R$ be the category of $R$-schemes on which $\zeta$ is locally nilpotent. We also set $\wh{\SpaceFl}_{\BP,R}:=\SpaceFl_\BP\whtimes_{\BaseOfD}\Spf R$ and $\wh{\SpaceFl}_\BP:=\wh{\SpaceFl}_{\BP,\BaseOfD\dbl\zeta\dbr}$. Let us consider the following functor
		
		\begin{eqnarray}\label{eqPAFlag}
			\ul{\Breve\CM}:&(\Nilp_R)^o &\longto  \Sets  \hspace{7cm}\vspace{-2mm}\\ \nonumber
			&\SSS &\longmapsto  \big\{\text{Isomorphism classes of }(\CL_+,\delta);\text{where: }\\ \nonumber  
			& &~~~~~~~~~~~~~~~~~~~~~~~~~~~~-\CL_+~\text{is an $L^+\BP$-torsor over $\SSS$ and}\\ \nonumber
			& &  ~~~~~~~~~~~~~~~~~~~~~~~~~~~~ -\text{a trivialization $\delta\colon  \CL \to LP_S$ of the}\\ \nonumber
			& & ~~~~~~~~~~~~~~~~~~~~~~~~~~~~~~~~~\text{associated loop torsors}
			\big\}. 
		\end{eqnarray}

		\begin{proposition}\label{PropFlRepUnBoundedRZ}
			The ind-scheme $\wh{\CF\ell}_{\BP,R}$ pro-represents the above functor.

		\end{proposition}
		
		\begin{proof}
			In order to illustrate how the representablity works, here we briefly sketch the proof, and we refer the reader to \cite[Theorem~4.4.]{AH_Local} for further details. We assume that $R=\BF\dbl\zeta\dbr$. Consider a pair $(\CL_+,\delta)\in \ul{\Breve\CM}(S)$. Choose an \fppf-covering $S' \to S$ which trivializes $\CL_+$, then the morphism $\delta$ is given by an element $g' \in LP(S')$. The image of the element $g' \in LP(S')$ under  $LP(S')\to\wh{\CF\ell}_\BP (S')$ is independent of the choice of the trivialization, and since $(\CL_+,\delta)$ is defined over $S$, it descends to a point $x \in \wh{\CF\ell}_\BP$. 
			
			Conversely let $x$ be in $\wh{\CF\ell}_\BP(S)$, for a scheme $S\in \Nilp_{\BF\dbl\zeta\dbr}$. The projection morphism $LP \to \CF\ell_\BP$ admits local sections for the \'etale topology by \cite[Theorem~1.4]{PR2}. Hence over an \'etale covering $S' \to S$ the point $x$ can be represented by an element $g'\in LP(S')$. We let $(\CL_+',\delta')=((L^+\BP)_{S'},g')$. It can be shown that it descends and gives $(\CL_+,\delta)$ over $S$.
		\end{proof}

		Now let us recall the notion of boundedness condition for local $\BP$-shtukas from \cite{AH_Local}. 

		\begin{definition}\label{DefEqClClosedInd}
			(a) For a finite extension $\BaseOfD\dbl\zeta\dbr\subset R\subset\BaseOfD\dpl\zeta\dpr^\alg$ of discrete valuation rings we consider closed ind-subschemes $\hat{Z}_R\subset\wh{\SpaceFl}_{\BP,R}$. We call two closed ind-subschemes $\hat{Z}_R\subset\wh{\SpaceFl}_{\BP,R}$ and $\hat{Z}'_{R'}\subset\wh{\SpaceFl}_{\BP,R'}$ \emph{equivalent} if there is a finite extension of discrete valuation rings $\BaseOfD\dbl\zeta\dbr\subset\wt R\subset\BaseOfD\dpl\zeta\dpr^\alg$ containing $R$ and $R'$ such that $\hat{Z}_R\whtimes_{\Spf R}\Spf\wt R \,=\,\hat{Z}'_{R'}\whtimes_{\Spf R'}\Spf\wt R$ as closed ind-subschemes of $\wh{\SpaceFl}_{\BP,\wt R}$.
			
			\medskip\noindent
			(b) Let $\hat{Z}=[\hat{Z}_R]$ be an equivalence class of closed ind-subschemes $\hat{Z}_R\subset\wh{\SpaceFl}_{\BP,R}$ and let $G_{\hat{Z}}:=\{\gamma\in\Aut_{\BaseOfD\dbl\zeta\dbr}(\BaseOfD\dpl\zeta\dpr^\alg)\colon \gamma(\hat{Z})=\hat{Z}\,\}$. We define the \emph{ring of definition $R_{\hat{Z}}$ of $\hat{Z}$} as the intersection of the fixed field of $G_{\hat{Z}}$ in $\BaseOfD\dpl\zeta\dpr^\alg$ with all the finite extensions $R\subset\BaseOfD\dpl\zeta\dpr^\alg$ of $\BaseOfD\dbl\zeta\dbr$ over which a representative $\hat{Z}_R$ of $\hat{Z}$ exists.
		\end{definition}
		
		For further explanation about the above definition see \cite[Remark~4.6 and Remark 4.7]{AH_Local}
		
		
		\begin{definition}\label{DefBDLocal}
			\begin{enumerate}
				\item \label{DefBDLocal_A}
				We define a \emph{bound} to be an equivalence class $\hat{Z}:=[\hat{Z}_R]$ of closed ind-subschemes $\hat{Z}_R\subset\wh{\SpaceFl}_{\BP,R}$, such that for all $R$ the ind-subscheme $\hat{Z}_R$ is stable under the left $L^+\BP$-action on $\SpaceFl_\BP$, and the special fiber $Z_R:=\hat{Z}_R\whtimes_{\Spf R}\Spec\kappa_R$ is a quasi-compact subscheme of $\SpaceFl_\BP\whtimes_{\BaseOfD}\Spec\kappa_R$. The ring of definition $R_{\hat{Z}}$ of $\hat{Z}$ is called the \emph{reflex ring} of $\hat{Z}$. Since the Galois descent for closed ind-subschemes of $\SpaceFl_\BP$ is effective, the $Z_R$ arise by base change from a unique closed subscheme $Z\subset\SpaceFl_\BP\whtimes_\BaseOfD\kappa_{R_{\hat{Z}}}$. We call $Z$ the \emph{special fiber} of the bound $\hat{Z}$. It is a projective scheme over $\kappa_{R_{\hat{Z}}}$ by Remark~\cite[Remark 4.3]{AH_Local} and \cite[Lemma~5.4]{H-V}, which implies that every morphism from a quasi-compact scheme to an ind-projective ind-scheme factors through a projective subscheme.
				\item \label{DefBDLocal_B}
				Let $\hat{Z}$ be a bound with reflex ring $R_{\hat{Z}}$. Let $\CL_+$ and $\CL_+'$ be $L^+\BP$-torsors over a scheme $S$ in $\Nilp_{R_{\hat{Z}}}$ and let $\delta\colon \CL\isoto\CL'$ be an isomorphism of the associated $L\genericG$-torsors. We consider an \'etale covering $S'\to S$ over which trivializations $\alpha\colon\CL_+\isoto(L^+\BP)_{S'}$ and $\alpha'\colon\CL'_+\isoto(L^+\BP)_{S'}$ exist. Then the automorphism $\alpha'\circ\delta\circ\alpha^{-1}$ of $(L\genericG)_{S'}$ corresponds to a morphism $S'\to L\genericG\whtimes_\BaseOfD\Spf R_{\hat{Z}}$. We say that $\delta$ is \emph{bounded by $\hat{Z}$} if for any such trivialization and for all finite extensions $R$ of $\BaseOfD\dbl\zeta\dbr$ over which a representative $\hat{Z}_R$ of $\hat{Z}$ exists the induced morphism 
				\[
				S'\whtimes_{R_{\hat{Z}}}\Spf R\to L\genericG\whtimes_\BaseOfD\Spf R\to \wh{\SpaceFl}_{\BP,R}
				\]
				factors through $\hat{Z}_R$. Furthermore we say that a local $\BP$-shtuka $(\CL_+, \tauLoc)$ is \emph{bounded by $\hat{Z}$} if the isomorphism $\tauLoc$ is bounded by $\hat{Z}$. 
			\end{enumerate}
		\end{definition}
		
		\begin{remark}\label{RemBound}
			Note that the condition of Definition~\ref{DefBDLocal}\ref{DefBDLocal_B} is satisfied for \emph{all} trivializations and for \emph{all} such finite extensions $R$ of $\BaseOfD\dbl\zeta\dbr$ if and only if it is satisfied for \emph{one} trivialization and for \emph{one} such finite extension \cite[Remark 4.9]{AH_Local}. 
		\end{remark}

		\subsection{Statement and the proof of the theorem}
		
		Before stating the main theorem, we need to define the notion of admissibility and speciality of $\FG$-shtukas ($\FG$-motives).
		
		\subsubsection{Admissible and Special $\FG$-Shtukas }\label{SubsectAdmandSpe}

		\begin{definition}\label{DefIwahori-Weyl}
			Assume that the generic fiber $\genericG$ of $\BP$ over $\Spec\BaseOfD\dpl z\dpr$ is connected reductive. Consider the base change $\genericG_L$ of $\genericG$ to $L=\BaseOfD^\alg\dpl z\dpr$. Let $S$ be a maximal split torus in $\genericG_L$ and let $T$ be its centralizer. Since $\BaseOfD^\alg$ is algebraically closed, $\genericG_L$ is quasi-split and so $T$ is a maximal torus in $\genericG_L$. Let $N = N(T)$ be the normalizer of $T$ and let $\CT^0$ be the identity component of the N\'eron model of $T$ over $\CO_L=\BaseOfD^\alg\dbl z\dbr$.
			
			The \emph{Iwahori-Weyl group} associated with $S$ is the quotient group $\wt{W}= N(L)\slash\CT^0(\CO_L)$. It is an extension of the finite Weyl group $W_0 = N(L)/T(L)$ by the coinvariants $X_\ast(T)_I$ under $I=\Gal(L^\sep/L)$:
			$$
			0 \to X_\ast(T)_I \to \wt W \to W_0 \to 1.
			$$
			By \cite[Proposition~8]{H-R1} there is a bijection
			\begin{equation}\label{EqSchubertCell}
				L^+\BP(\BaseOfD^\alg)\backslash L\genericG(\BaseOfD^\alg)/L^+\BP(\BaseOfD^\alg) \isoto \wt{W}^\BP  \backslash \wt{W}\slash \wt{W}^\BP
			\end{equation}
			where $\wt{W}^\BP := (N(L)\cap \BP(\CO_L))\slash \CT^0(\CO_L)$, and where $LP(R)=P(R\dbl z\dbr)$ and $L^+\BP(R)=\BP(R\dbl z\dbr)$ are the loop group, resp.\ the group of positive loops of $\BP$; see \cite[\S\,1.a]{PR2}, or \cite[\S4.5]{B-D}, \cite{Ngo-Polo} and \cite{Faltings03} when $\BP$ is constant.
			
		\end{definition}

		\begin{definition}\label{DefAdm}
			
			Let $\omega\in \wt{W}^\BP\backslash \wt{W}/\wt{W}^\BP$ and let $\BaseOfD_\omega$ be the fixed field in $\BaseOfD^\alg$ of $\{\gamma\in\Gal(\BaseOfD^\alg/\BaseOfD)\colon \gamma(\omega)=\omega\}$. There is a representative $g_\omega\in L\genericG(\BaseOfD_\omega)$ of $\omega$; see \cite[Example~4.12]{AH_Local}. The \emph{Schubert variety} $\CS(\omega)$ associated with $\omega$ is the ind-scheme theoretic closure of the $L^+\BP$-orbit of $g_\omega$ in $\SpaceFl_\BP\whtimes_{\BaseOfD}\BaseOfD_\omega$. It is a reduced projective variety over $\BaseOfD_\omega$. For further details see \cite{PR2} and \cite{Richarz}.

		\end{definition}

		\begin{definition}
			Define
			
			$$
			Adm(\hat{Z}_\nu)=\{\omega\in W_{K_\nu}\backslash W_\nu \slash W_{K_\nu} ; \CS(\omega)\subseteq Z_\nu\times{\kappa_{\hat{Z}_\nu}}\BF_q^\alg\}.
			$$
			
		\end{definition}

		\begin{definition}\label{ADLV}

			Let $\hat{Z}$ be a bound with reflex ring $R_{\hat{Z}}=\kappa\dbl\xi\dbr$ and special fiber $Z\subset\SpaceFl_\BP\whtimes_\BaseOfD\Spec\kappa$; see Definition~\ref{DefBDLocal}. Let $\ul\CL_0=(\CL_{0+},\hat{\tau}_0)$ be a local $\BP$-isoshtuka over a field $\BaseFldOfLocSht$ in $\Nilp_{\BaseOfD\dbl\zeta\dbr}$. Define the associated \emph{affine Deligne-Lusztig variety} $\Breve{X}_Z(\ul\CL_0)$ as the reduced closed ind-subscheme $\Breve{X}_Z(\ul\CL_0)\subset\SpaceFl_\BP\whtimes_\BaseOfD\Spec\BaseFldOfLocSht$ whose $K$-valued points (for any field extension $K$ of $\BaseFldOfLocSht$) are given by
			$$
			\Breve{X}_Z(\ul\CL_0)(K):=\big\{ (\ul\CL,\delta) \in \SpaceFl_\BP(K)\colon \delta^{-1}\,\tau_0\,\hat\sigma^\ast(\delta) \text{~induces a point in}~ Z(K)\big\},
			$$
			see Proposition \ref{PropFlRepUnBoundedRZ}. Similarly for an object $\hat{\ul\CM}$ in the category $\BP$-$\hat{\sigma}$-$\textbf{Cryst}_{\BF}(L)$ (resp. $\hat{\phi}$ in $\hat{\sigma}$-$\breve{\textbf{Cryst}}_{\BF_{\nu}}(\BF_q^{alg})$) we denote by $\Breve{X}_Z(\hat{\CM})$ (resp. $\Breve{X}_Z(\hat{\phi})$) the affine Deligne-Lusztig variety $\Breve{X}_Z(\ul\CL(\hat{\CM}))$ (resp. $\Breve{X}_Z(\ul\CL(\phi))$), where $\ul\CL(\hat{\CM})$   (resp. $\ul\CL(\phi)$) denote the corresponding local $\BP$-isoshtuka; See remark \ref{RemFunctPSHtToPisoSHt} b) and remark \ref{RemarkEquivOFMotCats}.  
			
			Note that when $\ul\BL=(LP,b\hat{\sigma}^*)$ is a trivialized local $P$-isoshtuka over a field $\BaseFldOfLocSht$ in $\Nilp_{\BaseOfD\dbl\zeta\dbr}$, the associated \emph{affine Deligne-Lusztig variety} $\Breve{X}_Z(\ul\BL)$ is the reduced closed ind-subscheme $\Breve{X}_Z(b)\subset\SpaceFl_\BP\whtimes_\BaseOfD\Spec\BaseFldOfLocSht$ whose $K$-valued points (for any field extension $K$ of $\BaseFldOfLocSht$) are given by
			$$
			\Breve{X}_Z(\ul\BL)(K)=\Breve{X}_Z(b)(K):=\big\{ g\in \SpaceFl_\BP(K)\colon g^{-1}\,b\,\hat\sigma^\ast(g) \in Z(K)\big\}.
			$$
			We set $\Breve{X}_{\preceq\omega}(b):=\Breve{X}_{\CS(\omega)}(b)$ if $Z$ is the affine Schubert variety $\CS(\omega)$ with $\omega \in \wt{W}$, see \cite{PR2}.
		\end{definition}

		\begin{definition}\label{Def_AdmLocalP-sht}
			
			Let $(\BP,\hat{Z})$ be a tuple of data consisting of the group $\BP$ with generic fiber $P$ and the bound $\hat{Z}$. We say that a local $P$-isoshtuka $\ul\CL$ is admissible if the associated affine Deligne-Lusztig variety $\Breve{X}_Z(\ul\CL)$ is nonempty. We say that $\ul\CL$ in the category of local $\BP$-shtukas (resp. $\hat{\ul\CM}$ in $\BP$-$\hat{\sigma}$-$\textbf{Cryst}_{\BF}(\BF_q^{alg})$, resp. $\hat{\phi}$ in $\BP\text{-}\hat{\sigma}\text{-}\breve{\textbf{Cryst}}_{\BF_{\nu}}(\BF_q^{alg})$)) is admissible if the corresponding object in the category of local $P$-isoshtukas is admissible.
		\end{definition}

		\begin{definition}[Special $\FG$-shtuka]\label{DefSpecialG-shtuka}
			\begin{enumerate}
				\item 
				A $\FG$-shtuka $\ul\CG$ over $S$ is called \emph{special} if its quasi-isogeny class has a representative in the image of $i_\ast: \nabla_n\scrH^1(C,\FT)(S)\to \nabla_n\scrH^1(C,\FG)(S)$ for some inclusion $i$ of a torus $\FT$ into $\FG$.
				\item
				Similarly, we say that the $G$-motive $\CM_{G}(-): \Rep_Q G\to \mot(L)$ is \emph{special} if there exist a $T$-motive $\CM_T(-): \Rep_Q T \to \mot(L)$ for a torus $T\subseteq G$, such that $\CM_{G}$ and $\CM_{T}\circ r$ are naturally isomorphic functors, where $r: \Rep_Q G\to \Rep_Q T$ is induced by the inclusion $T\subseteq G$.
				
			\end{enumerate}
		\end{definition}
		
		\begin{lemma}
			
			A $\FG$-shtuka $\ul\CG$ in $\nabla_n\scrH^1(C,\FG)(L)$ is special if and only if the associated $G$-motive under \ref{Eq_FunctShtToGMot} is special.
			
		\end{lemma}
		
		\begin{proof}
			The statement easily follows from the Tannakian formalism.
		\end{proof}
		
		\noindent

		Let $B(G)$ denote the $\sigma$-conjugacy classes in $G(\breve{Q})$, where $\breve{Q}=Q\otimes_{\BF_q}\BF_q^{\alg}$.  We recall that according to the Kottwitz description of $B(G)$ for local and global fields in terms of Galois gerbs in
		\cite{Ko5}, one assigns the two invariants $\nu_G$ and $\kappa_G$ two the elements of $B(G)$.
		Here we recall the following result of Hamacher and Kim \cite[Theorem 1.3]{H-K}.
		
		\begin{theorem}\label{ThmHamacher-Kim} For reductive group $G$ we have the following statements 
			\begin{enumerate}
				\item 
				Every $b\in B(G)$ is uniquely determined by Kottwitz invariants $\kappa_G(b)$ and $\nu_G(b)$.
				
				\item
				When $G$ is quasi-split, the canonical map
				\[
				\bigcup_{~~~~~~~~~~T\subseteq G ~ \text{max. torus}}B(T)\to B(G)
				\]
				is surjective. 
			\end{enumerate}

		\end{theorem}

		\begin{remark}\label{Rem_UnBD_RZ_space}
			
			For a scheme $\SSS$ in $\Nilp_{\BaseOfD\dbl\zeta\dbr}$ let $\bar{\SSS}$ denote the closed subscheme $\Var_\SSS(\zeta)\subseteq \SSS$. On the other hand for a scheme $\bar \TTT$ over $\BaseOfD$ we set $\wh{\TTT}:=\bar \TTT\whtimes_{\Spec\BaseOfD}\Spf\BaseOfD\dbl\zeta\dbr$. Then $\wh \TTT$ is a $\zeta$-adic formal scheme with underlying topological space $\bar \TTT=\Var_{\wh \TTT}(\zeta)$.
			Recall that to a given local $\BP$-shtuka $\ul\BL$ over an $\BaseOfD$-scheme $\bar{\TTT}$ we associate the functor
			\begin{eqnarray*}
				\Breve{\CM}_{\ul\BL}\colon  \Nilp_{\wh{\TTT}} &\to&  \Sets\\
				\SSS &\mapsto & \big\{\text{Isomorphism classes of }(\ul{\CL},\bar{\delta})\colon\;\text{where }\ul{\CL}~\text{is a local $\BP$-shtuka}\\ 
				&&~~~~ \text{over $\SSS$ and }\bar{\delta}\colon  \ul{\CL}_{\bar{\SSS}}\to \ul\BL_{\bar{\SSS}}~\text{is a quasi-isogeny  over $\bar{\SSS}$}\big\}. 
			\end{eqnarray*}
			Here we say that $(\ul\CL,\bar\ppsi)$ and $(\ul\CL',\bar\ppsi')$ are isomorphic if $\bar\ppsi^{-1}\circ\bar\ppsi'$ lifts to an isomorphism $\ul\CL'\to\ul\CL$. Note that the group $\QIsog_{\bar \TTT}(\ul\BL)$ of quasi-isogenies of $\ul\BL$ acts on the functor $\Breve{\CM}_{\ul\BL}$ via $j\colon(\ul\CL,\bar\delta)\mapsto(\ul\CL,j\circ\bar\delta)$ for $j\in\QIsog_{\bar \TTT}(\ul\BL)$.
			
			Recall from \cite[Proposition~4.11]{AH_Local} that the above functor $\Breve{\CM}_{\ul\BL}$ is represented by an ind-scheme, ind-quasi-projective, hence ind-separated and of ind-finite type over $\wh \TTT=\bar \TTT\whtimes_{\BaseOfD}\Spf\BaseOfD\dbl\zeta\dbr$. If $\BP$ is parahoric, then $\Breve{\CM}_{\ul\BL}$ is ind-projective over $\wh \TTT$. Moreover, if $\ul\BL$ is trivialized by an isomorphism $\alpha\colon\ul\BL\isoto\bigl((L^+\BP)_{\bar \TTT},b\hat\sigma^*\bigr)$ over $\bar\TTT$ with $b\in L\genericG(\bar \TTT)$ then $\Breve{\CM}_{\ul\BL}$ is represented by the ind-scheme $\wh\SpaceFl_{\BP,\wh\TTT}:=\SpaceFl_\BP\whtimes_\BaseOfD\wh \TTT$.
			
		\end{remark}

		\subsubsection{$\nabla\scrH$-data and the functor $\CL^\ast(-)$}

		Let us first recall the following definition for the level $H$-structure, for a compact open subgroup $H\subseteq \FG(\BA^\ul\nu)$.

		\begin{definition}\label{level structure}
			Let $\omega^\circ\colon \Rep_{\BO^{\ul\nu}}\FG \to \FM od_{\BA^{\ul\nu}}$ denote the forgetful functor. For a global $\FG$-shtuka $\ul\CG$ over $S$, consider the set of isomorphisms of tensor functors $\Isom^{\otimes}(\omega_{\BO^{\ul\nu}}^\circ,\omega^\ul\nu(\ul{\CG}))$. Note that this is non-empty, see \cite[Lemma~6.2]{AH_Global} and $\FG(\BA^{\ul\nu})=\Aut^\otimes(\omega^\circ)$. The set $\Isom^{\otimes}(\omega_{\BO^{\ul\nu}}^\circ,\omega^\ul\nu(\ul{\CG}))$ admits an action of $\pi_1^\et(S,\bar{s})\times\FG(\BA^{\ul\nu})$ where $\FG(\BA^{\ul\nu})$ acts through $\omega^\circ$ and $\pi_1^\et(S,\bar{s})$ acts through $\omega^\ul\nu(\ul\CG)$. For a compact open subgroup $H\subseteq \FG(\BA^{\ul\nu})$ we define a \emph{rational $H$-level structure} $\bar\gamma$ on a global $\FG$-shtuka $\ul\CG$ over a connected scheme $S\in\Nilp_{A_{\ul\nu}}$ as a $\pi_1^\et(S,\bar{s})$-invariant $H$-orbit $\bar\gamma=\gamma H$ in $\Isom^{\otimes}(\omega_{\BO^{\ul\nu}}^\circ,\omega^\ul\nu(\ul{\CG}))$. For a non-connected scheme $S$ we make a similar definition choosing a base point on each connected component and a rational $H$-level structure on the restriction to each connected component separately. We denote by $\nabla_n^H\scrH^1(C,\FG)^{\ul\nu}$ the category fibered in groupoids over $\Nilp_{A_\ul\nu}$ whose $S$-valued points $\nabla_n^H\scrH^1(C,\FG)^{\ul\nu}(S)$ is the category whose objects are tuples $(\ul\CG,\gamma H)$, consisting of a global $\FG$-shtuka $\ul\CG$ in $\nabla_n\scrH^1(C,\FG)^{\ul\nu}(S)$ together with a rational $H$-level structure $\gamma H$, and whose morphisms are quasi-isogenies of global $\FG$-shtukas that are isomorphisms at the characteristic places $\nu_i$ and are compatible with the $H$-level structures.

		\end{definition}

		\begin{definition}\label{DefNablaHdata}
			\begin{enumerate}
				\item
				Fix an $n$-tuple $\ul \nu:=(\nu_i)$ of closed points of $C$.
				A \emph{$\nabla\scrH^{\ul \nu}$-data} is a tuple $(\FG, (\hat{Z}_{\nu_i})_i, H)$ consisting of
				\begin{enumerate}
					\item
					a parahoric group scheme $\FG$ over $C$,
					\item
					an $n$-tuple $(\hat{Z}_{\nu_i})_i$ of bounds with reflex rings $R_{\hat{Z}_{\nu_i}}$,
					\item
					a compact open subgroup $H\subseteq \FG(\BA_Q^{\ul \nu})\cong Aut^\otimes (\omega_{\BA^{\ul \nu}}^\circ)$.
				\end{enumerate}
				We use the shorthand $\nabla\scrH$-data when $\ul\nu$ is clear from the context.  
				A morphism between two $\nabla\scrH$-data $(\FG, (\hat{Z}_{\nu_i}), H)$ and $(\FG', (\hat{Z}_{\nu_i}'), H')$ is a morphism $\phi: \FG\to \FG'$ such that the morphism $\hat{Z}_{\nu_i}\to\wh{\SpaceFl}_{\BP_{\nu_i}'}$, obtained by the inclusion $\hat{Z}_{\nu_i}\to \wh{\SpaceFl}_{\BP_{\nu_i}}$ followed by the induced morphism $\wh{\SpaceFl}_{\BP_{\nu_i}}\to\wh{\SpaceFl}_{\BP_{\nu_i}'}$, factors through $\hat{Z}_{\nu_i}'$, for every $i$, and $\phi(H)\subseteq H'$.
				
				\item \label{DefNablaHdataRatLevelStr}
				For a compact open subgroup $H\subseteq \FG(\BA_Q^{\ul \nu})$ we define a \emph{rational $H$-level structure} $\bar\gamma$ on a global $\FG$-shtuka $\ul \CG$ over $S\in\Nilp_{A_{\ul\nu}}$ as a $\pi_1(S,\bar{s})$-invariant $H$-orbit $\bar\gamma=\gamma H$ in $\Isom^{\otimes}(\omega^\circ,\omega^\ul\nu(\ul{\CG}))$.
				\item 
				We denote by $\nabla_n^H\scrH^1(C,\FG)^{\ul \nu}$ the category fibered in groupoids whose $S$-valued points $\nabla_n^H\scrH^1(C,\FG)^{\ul \nu}(S)$ is the category whose objects are tuples $(\ul \CG,\bar\gamma)$, consisting of a global $\FG$-shtuka $\ul\CG$ in $\nabla_n \scrH^1(C,\FG)^{\ul\nu}(S)$ together with a rational $H$-level structure $\bar\gamma$, and whose morphisms are quasi-isogenies of global $\FG$-shtukas (see Definition~\ref{DefGlobal_Sht}\ref{DefGlobal_Sht_Mor}) that are isomorphisms at the characteristic places $\nu_i$ and are compatible with the $H$-level structures.  
				

				\item
				Fix an $n$-tuple $\ul\nu=(\nu_i)$ of places on the curve $C$ with $\nu_i\ne\nu_j$ for $i\ne j$ and let $\nabla\scrH^1(C,\FG)^{{\ul\nu}}$ denote the formal completion of the stack $\nabla\scrH^1(C,\FG)$ at $\ul\nu$. Let  $\hat{Z}_{\ul\nu}:=(\hat{Z}_{\nu_i})_i$ be a tuple of closed ind-subschemes $\hat{Z}_i$ of $\wh{\SpaceFl}_{\BP_{\nu_i}}$ which are bounds in the sense of Definition~\ref{DefBDLocal}.
				Let $\ul\CG$ be a global $\FG$-shtuka in $\nabla_n\scrH^1(C,\FG)^{{\ul\nu}}(S)$. 
				We say that \emph{$\ul\CG$ is bounded by $\hat{Z}_{\ul\nu}:=(\hat{Z}_{\nu_i})_i$} if for every $i$ the associated local $\BP_{\nu_i}-$shtuka $\omega_{\nu_i}(\ul\CG)$ is bounded by $\hat{Z}_{\nu_i}$.

			\end{enumerate}
		\end{definition}

		To a $\nabla\scrH^{\ul \nu}$-data $(\FG,(\hat{Z}_{\nu_i})_i, H)$ we associate a moduli stack $\nabla_n^{H,\hat{Z}_{\ul\nu}}\scrH^1(C,\FG)^{{\ul\nu}}$ parametrizing $\FG$-shtukas bounded by $\hat{Z}_{\ul\nu}$ at place $\nu$ which are additionally equipped with a level $H$-structure. Furthermore this correspondence is functorial, i.e. a morphism $\phi:(\FG, (\hat{Z}_i), H)\to(\FG', (\hat{Z}_i'), H')$ between $\nabla\scrH$-data induces the following morphism 
		\begin{equation}\label{EqSHtDataAssignment}
			\phi_\ast: \nabla_n^{H,\hat{Z}_{\ul\nu}}\scrH^1(C,\FG)^{{\ul\nu}}\to\nabla_n^{H',\hat{Z}_{\ul\nu}'}\scrH^1(C,\FG')^{{\ul\nu}}
		\end{equation}

		\begin{definition}\label{Def_AdmG-sht}
			Fix a $\nabla\scrH$-data $(\FG,(\hat{Z}_{\nu_i})_i)$. We say that the global $\FG$-shtuka $\ul\CG:=(\CG,\tau)$ in $\nabla_n^{H,\hat{Z}_{\ul\nu}}\scrH^1(C,\FG)^{{\ul\nu}}$ is admissible if the associated local $\BP_{\nu_i}$-shtukas $\omega_{\nu_i}(\ul\CG)$, 
			at all characteristic places $\nu_i$ for $i=1,\cdots ,n$, are admissible.
			
		\end{definition}

		\begin{proposition}\label{PropHE}
			Let $B(P_\nu)$ be the set of $\hat{\sigma}$-conjugacy classes of elements in $P(\breve{Q}_\nu)$. Let $B(\wt{W}_\nu)$ be the set of $\hat{\sigma}$-conjugacy classes in the Iwahori-Weyl group $\wt W_\nu:=\wt W(P_\nu,S_\nu)$. Let $\Psi_\nu:  B(\wt W_\nu)\to B(P_\nu)$ denote the natural morphism induced from the natural inclusion $N(K_\nu)\to G(K_\nu)$.
			A global $\FG$-shtuka $\ul\CG \in \nabla_n^{\ul w}\scrH^1(C,\FG)^{\ul \nu}(\BaseOfD)$ is admissible if the $\hat{\sigma}$-conjugacy class $[b_i]:=[\omega_{\nu_i}(\tau_\ul\CG)]\in B(P_{\nu_i})$ lies in $\bigcup_{w'\preceq w_i}\Psi(w')$, for all  $1\leq i\leq n$. Here $\nabla_n^{\ul w}\scrH^1(C,\FG)^{\ul \nu}$ denotes the stack  $\nabla_n^{\hat{Z}_{\ul \nu}}\scrH^1(C,\FG)^{\ul \nu}$ with $\hat{Z}_{\ul \nu}:=(\CS(w_i)\wh \times_{\BF_{\omega_i}}\Spf \BF_{\omega_i}\dbl \zeta \dbr)_i$.
		\end{proposition}

		\begin{proof}
			This follows from \cite[Theorem 2.1]{He} and Lemma \ref{LemPSHtLocTrivForEtTop}. 
		\end{proof}
		
		\begin{definition}
			
			We define the following functorial assignment
			$$
			\nabla_n^{*,-}\scrH^1(C,-)^{{\ul\nu}}:(\FG, (\hat{Z}_{\nu_i})_i)\mapsto \nabla_n^{*,\hat{Z}_{\ul\nu}}\scrH^1(C,\FG)^{{\ul\nu}}:=\invlim[H]\nabla_n^{H,\hat{Z}_{\ul\nu}}\scrH^1(C,\FG)^{{\ul\nu}},
			$$ 
			where $H$ runs over all compact open subgroups of $\FG(\BA_Q^{\ul \nu})\cong Aut^\otimes (\omega_{\BA^{\ul \nu}}^\circ)$.
			
		\end{definition}

		\forget{
			--------------------
			\begin{definition}\label{DefFunctor_L} Keep the above notation. Define the functor $\CL$ (resp. $\CL_{spe}$, resp. $\CL_{adm}$) which sends a $\nabla\scrH-$data $(\FG, (\hat{Z}_{\nu_i})_i, H)$ (resp. $(\FG, (\hat{Z}_{\nu_i})_i)$) to the disjoint union $\bigsqcup_{\ul\CG} \scrS(\ul\CG)$ of the quotient stack 
				
				\begin{equation}\label{EqSetS}
					\scrS(\ul\CG):=I(Q)(\ul \CG) \big{\backslash}\prod_{\nu_i} \Breve{X}_{Z_{\nu_i}}(\omega_{\nu_i}(\ul \CG))\times X^{\ul \nu}(\ul\CG) 
					{\slash} H,
				\end{equation}
				
				\noindent
				where $\ul \CG$ runs over the quasi-isogeny classes (resp. quasi-isogeny classes of special, resp. quasi-isogeny classes of admissible) $\FG$-shtukas and $X^{\ul \nu}(\ul\CG)=\Isom^{\otimes}(\omega^\circ,\omega^\ul\nu(\ul{\CG}))$.
				We define the functor $\CL^*$ (resp. $\CL_{spe}^*$, resp. $\CL_{adm}^*$) as the functor which sends $(\cG, (\hat{Z}_i)_i)$ to $\invlim[H]\CL(\cG, (\hat{Z}_{\nu_i})_i, H)$ (resp. $\invlim[H]\CL_{spe}(\cG, (\hat{Z}_{\nu_i})_i, H)$, resp. $\invlim[H]\CL_{adm}(\cG, (\hat{Z}_{\nu_i})_i, H)$).\comment{switch to the definition bellow}
				
			\end{definition}
			----------------
		}
		
		\begin{definition}\label{DefFunctor_L} Keep the above notation. We define the functor $\CL$ as the functor that sends a $\nabla\scrH-$data $(\FG, (\hat{Z}_{\nu_i})_i, H)$ to the disjoint union $\bigsqcup_{\phi} \scrS(\phi)$ of the quotient stack 
			
			\begin{equation}
				\scrS(\phi):=\Aut(\phi) \big{\backslash}\prod_{\nu_i} \Breve{X}_{Z_{\nu_i}}(\omega_{\nu_i}(\phi))\times X^{\ul \nu} (\phi)
				{\slash} H,
			\end{equation}
			
			\noindent
			where $\phi$ runs over isomorphism classes $\phi$ in $\FG\text{-}\breve{\CM ot} (\BF_q^{\alg})$ and $X^{\ul \nu}(\phi)=\Isom^{\otimes}(\omega^\circ,\omega^\ul\nu(\phi))$.
			
			We use the notation $\CL_{adm}$ (resp. $\CL_{spe}$) when we require that $\phi$ in the index set is in addition geometric and admissible (resp. special). We recall that $\phi$ in $\FG-\breve{\CM ot} (\BF_q^{\alg})$ is called geometric if the associated $\FG$-motive $\CM_\phi$ lies in the essential image of the functor \ref{Eq_FunctShtToGMot}. It is called admissible (resp. special) if the associated $\FG$-motive $\CM_\phi$ is admissible (resp. special), see Definitions \ref{Def_AdmG-sht} and \ref{DefSpecialG-shtuka}.   
			
			We define the functor $\CL^*$ (resp. $\CL_{spe}^*$, resp. $\CL_{adm}^*$) as the functor which sends $(\FG, (\hat{Z}_{\nu_i})_i)$ to $\invlim[H]\CL(\FG, (\hat{Z}_{\nu_i})_i, H)$ (resp. $\invlim[H]\CL_{spe}(\FG, (\hat{Z}_{\nu_i})_i, H)$, resp. $\invlim[H]\CL_{adm}(\FG, (\hat{Z}_{\nu_i})_i, H)$).
			
		\end{definition}

		\begin{remark}\label{RemGeometricisGroupTheoretic}
			Note that there is a group theoretic criterion for a $\FG$-motive to be geometric. Namely for a place $\nu$ in $C$ a $\FG$-motive $\wt\CM(-)$ gives rise to a $\BP_\nu$-$\hat{\sigma}$-crystal, and equivalently to a  local $\BP_\nu$-shtuka $\ul\BL_\nu=(L^+\BP,b_\nu\hat{\sigma}^\ast)$, which gives an element in $\hat{\sigma}_\nu$- conjugacy class $B(P_\nu)$. Assume that it is basic and trivial for almost all places $\nu$ in  $C$ and satisfies $\sum_{\nu} inv(b_\nu)=0$. Then there is a global $\ul\CG$ such that $\wh{\Gamma}_\nu(\ul\CG)=\ul\BL_\nu$ by \cite[chapitre 4]{NgoDac}. Notice that we have an isomorphism of functors $\CM(\ul\CG)\to\wt\CM$, by \cite[Prop. 5.2]{CMot}.
			
		\end{remark}

		\subsubsection{The Main Theorem}

		Now we can state the main theorem:
		
		\begin{theorem} \label{ThmL-RForG-ShtTXT}
			
			There exist a canonical $\FG(\BA^{\ul\nu})\times Z(Q)$-equivariant isomorphism of functors
			
			\begin{equation*}
				\CL^*(-)\tilde{\to} \nabla_n^{*,-}\scrH^1(C,-)^{\ul\nu}(\ol \BF).
			\end{equation*}

			Moreover there is an operation $\ul\Phi_m$ on the left hand side of the above isomorphism that corresponds to the Frobenius endomorphism $\sigma^m$ on the right hand side. Furthermore:\\
			
			\begin{enumerate}
				\item 
				The functor $\CL^\ast(-)$ can be replaced by $\CL_{adm}^\ast(-)$. Note that for tamely ramified $\FG$ there exists a group theoretic description for $\CL_{adm}^\ast(-)$, moreover, \forget{due to Xuhua He when $X_{\underline{v}} ({\cal{G}}') \not= \emptyset$}
				
				\item
				When $\FG$ is quasi-split, one may further replace $\CL^\ast(-)$ by $\CL_{spe}^\ast(-)$.
				
			\end{enumerate}

		\end{theorem}

		\begin{proof} The proof consists of the following parts. In the first part we understand the points of a newton stratum of $\nabla_n^{H,\hat{Z}_{\ul\nu}}\scrH^1(C,\FG)^{{\ul\nu}}$ using the uniformization theory. Below we recall this part from \cite{AH_Global}. We in particular check that the isomorphism is compatible with the $\FG(\BA^\ul\nu)\times Z(Q)\times \ul\Phi_m$ action; see Proposition \ref{PropThetaisEquivariantUnderGZPhi}. In the second part one needs to understand the index set, see definition \ref{DefFunctor_L}. To achieve this we need to describe the index set in terms of the representations of the motivic fundamental groupoid, which we discussed in Section \ref{SectG-ShtAndGMotives} and determine the admissible and special ones in the group theoretic sense; See subsection \ref{SubsectAdmandSpe}. Let us first recall the following ingredients.

			\begin{remark}
				
				Let $\genericG$ denote the generic fiber of $\BP$ and let $b\in L\genericG(\BaseFldOfLocSht)$ for some field $\BaseFldOfLocSht\in\Nilp_{\BaseOfD\dbl\zeta\dbr}$. 
				With $b$ Kottwitz associates a slope homomorphism
				$$
				\nu_b\colon  D_{\BaseFldOfLocSht\dpl z\dpr}\to \genericG_{\BaseFldOfLocSht\dpl z\dpr},
				$$
				called Newton polygon of $b$; see \cite[4.2]{Ko1}. Here $D$ is the diagonalizable pro-algebraic group over ${\BaseFldOfLocSht\dpl z\dpr}$ with character group $\BQ$. The slope homomorphism is characterized by assigning the slope filtration of $(V\otimes_{\BaseOfD\dpl z\dpr}{\BaseFldOfLocSht\dpl z\dpr},\rho(b)\cdot(\id\otimes \hat{\sigma}))$ to any $\BaseOfD\dpl z\dpr$-rational representation $(V,\rho)$ of $\genericG$; see \cite[Section 3]{Ko1}.  
				We assume that $b\in L\genericG(\BaseFldOfLocSht)$ satisfies a \emph{decency equation for a positive integer $s$}, that is,
				\begin{equation}\label{EqDecency}
					(b\hat{\sigma})^s\;=\;s\nu_b(z)\,\hat{\sigma}^s\qquad\text{in}\quad L\genericG(\BaseFldOfLocSht)\ltimes \langle\hat{\sigma}\rangle.
				\end{equation}
				
			\end{remark}
			\begin{remark}\label{RemDecent}
				Assume that $b\in L\genericG(\BaseFldOfLocSht)$ is decent with the integer $s$ and let $\ell\subset\BaseFldOfLocSht^\alg$ be the finite field extension of $\BaseOfD$ of degree $s$. Then $b\in L\genericG(\ell)$ because by \eqref{EqDecency} the element $b$ has values in the fixed field of $\hat\sigma^s$ which is $\ell$. Note that if $\BaseFldOfLocSht$ is algebraically closed, any $\hat{\sigma}$-conjugacy class in $L\genericG(\BaseFldOfLocSht)$ contains an element satisfying a decency equation; see \cite[Section 4]{Ko1}. 
			\end{remark}
			
			\begin{remark}\label{RemJb}
				With the element $b\in L\genericG(\BaseFldOfLocSht)$ one can associate a connected algebraic group $J_b$ over $\BaseOfD\dpl z\dpr$ which is defined by its functor of points that assigns to an $\BaseOfD\dpl z\dpr$-algebra $R$ the group
				$$
				J_b(R):=\big\{g \in \genericG(R\otimes_{\BaseOfD\dpl z\dpr} {\BaseFldOfLocSht\dpl z\dpr})\colon g^{-1}b\hat{\sigma}(g)=b\big\}.
				$$
				Let $b$ satisfy a decency equation for the integer $s$ and let $F_s$ be the fixed field of $\hat{\sigma}^s$ in ${\BaseFldOfLocSht\dpl z\dpr}$. Then $\nu_b$ is defined over $F_s$ and $J_b\times_{\BaseOfD\dpl z\dpr}F_s$ is the centralizer of the 1-parameter subgroup $s\nu_b$ of $\genericG$ and hence a Levi subgroup of $\genericG_{F_s}$; see \cite[Corollary 1.9]{RZ}. In particular $J_b(\BaseOfD\dpl z\dpr)\subset \genericG(F_s)\subset L\genericG(\ell)$ where $\ell$ is the finite field extension of $\BaseOfD$ of degree $s$.
			\end{remark}

			\begin{point}\textbf{Rapoport-Zink Spaces for local $\BP$-shtukas}\label{PointRZSpece} Let $\hat{Z}$ be a bound with reflex ring $R_{\hat{Z}}=\kappa\dbl\xi\dbr$ and special fiber $Z\subset\SpaceFl_\BP\whtimes_\BaseOfD\Spec\kappa$ ; see Definition~\ref{DefBDLocal}. Let $\ul\BL_0=(L^+\BP,b\hat{\sigma}^*)$ be a trivialized local $\BP$-shtuka over a field $\BaseFldOfLocSht$ in $\Nilp_{\BaseOfD\dbl\zeta\dbr}$. Assume that $b$ is decent with integer $s$ and let $\ell\subset\BaseFldOfLocSht^\alg$ be the compositum of the residue field $\kappa$ of $R_{\hat{Z}}$ and the finite field extension of $\BaseOfD$ of degree $s$. Then $b\in L\genericG(\ell)$ by Remark~\ref{RemDecent}. So $\ul\BL_0$ is defined over $\ell$ and we may replace $\BaseFldOfLocSht$ by $\ell$. Note that $\ell\dbl\xi\dbr$ is the unramified extension of $R_{\hat{Z}}$ with residue field $\ell$.
				
				set $\ul\CL_0:=\ul\BL_0$ and consider the following sub-functor of \ref{eqPAFlag}

				\begin{eqnarray}\label{EqR-ZFunct}
					\ul{\Breve{\CM}}_{\ul{\BL}_0}^{\hat{Z}}\colon (\Nilp_{\BaseFldOfLocSht\dbl\xi\dbr})^o &\longto & \Sets\\\nonumber
					S&\longmapsto & \Bigl\{\,\text{Isomorphism classes of }(\ul\CL,\bar\ppsi) \text{where}\\\nonumber
					&&~~~~~~~\ul\CL \text{~is a local} ~\BP\text{-shtuka} \text{~bounded by $\hat{Z}$} \text{and} \\\nonumber
					&&~~~~~~~~~~~~~~~~~~\ol\delta:\CL_{\ol S}\to \ul\BL_{0, \ol S} \text{~is a quasi-isogeny}\Bigr\}
				\end{eqnarray}
				
				Recall that in \cite[Theorem 4.18]{AH_Local} we proved the following
				\begin{theorem}
					
					The above functor $\ul{\Breve{\CM}}_{\ul{\BL}_0}^{\hat{Z}}\colon (\Nilp_{\BaseFldOfLocSht\dbl\xi\dbr})^o\to\Sets$ is pro-representable by a formal scheme over $\Spf \BaseFldOfLocSht\dbl\xi\dbr$, which is called a \emph{bounded Rapoport-Zink space for local $\BP$-shtukas}. It is locally formally of finite type. Its underlying reduced subscheme equals $\Breve{X}_{Z}(b)$. In particular $\Breve{X}_Z(b)$ is a scheme locally of finite type over $\BaseFldOfLocSht$.

				\end{theorem}
				
				
			\end{point}
			
			\begin{point}\textbf{Uniformization Theory for the moduli of global $\FG$-shtukas} In this paragraph we recall the uniformization theory of the moduli stacks of global $\FG$-shtukas from \cite{AH_Global}.
				\begin{lemma}\label{LemLocalIsogeny}
					Let $\ul\CG\in\nabla_n\scrH^1(C,\CG)^{\ul\nu}(\SSS)$ be a global $\FG$-shtuka over $\SSS$. Let $\ul\CL_{\nu_i}(\ul\CG)$ be the local $\BP_{\nu_i}$-shtuka associated with $\ul\CG$ via the crystalline realization functor $\omega_{\nu_i}(-)$ at the characteristic place $\nu_i$. Let $\tilde f\colon\ul{\tilde\CL}{}'_{\nu_i}\to \ul\CL_{\nu_i}(\ul\CG)$ be a quasi-isogeny of local $\BP_{\nu_i}$-shtukas over $\SSS$. Then there exists a unique global $\FG$-shtuka $\ul\CG'\in\nabla_n\scrH^1(C,\FG)^{\ul\nu}(\SSS)$ and a unique quasi-isogeny $g=g_{\tilde{f}}\colon\ul\CG'\to\ul\CG$ which is the identity outside $\nu_i$, such that the local $\BP_{\nu_i}$-shtuka associated with $\ul\CG'$ is $\ul{\tilde\CL}{}'_\nu$ and the quasi-isogeny of local $\BP_\nu$-shtukas induced by $g$ is $\tilde f$. We denote $\ul\CG'$ by $\tilde f^*\ul\CG$.
				\end{lemma}
				
				\begin{proof}
					It is a particular case of \cite[Proposition~5.7]{AH_Local}.
				\end{proof}
				
				Let $H\subset \FG (\BA_Q^{\ul\nu})$ be a compact open subgroup. Let $\ul{\CG}_0$ be a global $\FG$-shtuka over an algebraically closed field $\BaseFldInSectUnif$ with characteristic $\ul \nu$, bounded by $\hat{Z}_{\ul\nu}$. We let $I_{\ul\CG_0}(Q)$ denote the group $\QIsog_\BaseFldInSectUnif(\ul\CG_0)$ of quasi-isogenies of $\ul\CG_0$; see Definition~\ref{DefGlobal_Sht}\ref{DefGlobal_Sht_Mor}. Let $(\ul{\BL}_i)_{i=1\ldots n}:=\omega_\ul\nu(\ul\CG_0)$ denote the associated tuple of local $\BP_{\nu_i}$-shtukas over $\BaseFldInSectUnif$, where $\omega_{\ul\nu}$ is the crystalline realization functor from proposition~\ref{PropEquivOfMotCats} b). Let $\Breve{\CM}_{\ul \BL_i}^{\hat{Z}_i}$ denote the Rapoport-Zink space of $\ul\BL_i$. Since $\BaseFldInSectUnif$ is algebraically closed we may assume that all $\ul\BL_i$ are trivialized and decent. The product $\prod_i \Breve{\CM}_{\ul{\BL}_i}^{\hat{Z}_i}:=\Breve{\CM}_{\ul{\BL}_1}^{\hat{Z}_1}\whtimes_\BaseFldInSectUnif\ldots\whtimes_\BaseFldInSectUnif\Breve{\CM}_{\ul{\BL}_n}^{\hat{Z}_n}$ is a formal scheme locally formally of finite type over $\Spec\BaseFldInSectUnif\whtimes_{\BF_{\ul\nu}}\Spf A_{\ul\nu}=\Spf\BaseFldInSectUnif\dbl\zeta_1,\ldots,\zeta_n\dbr=:\Spf\BaseFldInSectUnif\dbl\ul\zeta\dbr$. 
				Using the above proposition \ref{LemLocalIsogeny}, one may construct the following morphism 
				\begin{equation}\label{EqUnifMorphismI}
					\Theta_{\ul{\cG}_0}'\colon  \prod_i \Breve{\CM}_{\ul\BL_i}^{\hat{Z}_{\nu_i}}\times X^\ul\nu(\ul\cG_0) {\slash}H \longto \nabla_n^{H,\hat{Z}_{\ul\nu}}\scrH^1(C,\FG)^{{\ul\nu}}\whtimes_{\BF_{\ul\nu}}\Spec\BaseFldInSectUnif
				\end{equation}
				
				which is defined by sending $(\ul\CL_{\nu_i},\delta_i)_i \times \alpha H$ to 
				$(\delta_n^\ast\dots\delta_1^\ast\ul\cG_0, \omega^\ul\nu(g_{\delta_n})^{-1}\circ\dots \circ \omega^\ul\nu(g_{\delta_1})^{-1}\circ \alpha)$. Here $X^{\ul \nu}(\ul\CG_0)=\Isom^{\otimes}(\omega^\circ,\omega^\ul\nu(\ul{\CG}_0))$. For the sake of completeness we recall the following important observation from \cite{AH_Global}.

				\begin{remark}\label{RemQuotientByI}
					Recall that a formal algebraic Deligne-Mumford stack $\CX$ over $\Spf\breve R_{\ulHZ}$ (see \cite[Definition~A.5]{Har1}) is \emph{$\CJ$-adic} for a sheaf of ideals $\CJ\subset\CO_\CX$, if for some (any) presentation $X\to\CX$ the formal scheme $X$ is $\CJ\CO_X$-adic, that is, $\CJ^r\CO_X$ is an ideal of definition of $X$ for all $r$. We then call $\CJ$ an \emph{ideal of definition of $\CX$}. We say that $\CX$ is \emph{locally formally of finite type} if $\CX$ is locally noetherian, adic, and if the closed substack defined by the largest ideal of definition (see \cite[A.7]{Har1}) is an algebraic stack locally of finite type over $\Spec\BaseFldInSectUnif$.
					The quotient of $\prod_i \Breve\CM_{\ul\BL_i}^{\hat{Z}_i}\times X^{\ul \nu}(\ul\CG_0)/H$ by the abstract group $I_{\ul\CG_0}\!(Q)$ exists as a locally noetherian, adic formal algebraic Deligne-Mumford stack locally formally of finite type over $\Spf\breve R_{\ulHZ}$; see \cite[Proposition 7.7]{AH_Global}. \forget{ and is of the form
						\begin{equation}\label{EqSourceOfTheta}
							I_{\ul\CG_0}\!(Q) \big{\backslash}\bigl(\prod_i \Breve\CM_{\ul\BL_i}^{\hat{Z}_i}\times X^{\ul \nu}(\ul\CG_0)/H\bigr) \enspace\cong\enspace \coprod_{\bar\gamma} \Gamma_{\bar\gamma}\big{\backslash}\prod_i \Breve\CM_{\ul\BL_i}^{\hat{Z}_i}\,.
						\end{equation}
						Here $\bar\gamma:=\gamma H\in X^{\ul \nu}(\ul\CG_0)/H$ runs through a set of representatives for the countable double coset $$I_{\ul\CG_0}\!(Q) \backslash X^{\ul \nu}(\ul\CG_0)/H\cong \epsilon\bigl(I_{\ul\CG_0}\!(Q)\bigr) \backslash \FG(\BA^{\ul\nu})/H,$$ and 
						$$
						\Gamma_{\bar\gamma}\;:=\; I_{\ul\CG_0}\!(Q)\cap \bigl(\prod_i J_{\ul{\BL}_i}(Q_{\nu_i})\times \gamma H \gamma^{-1}\bigr)\;\subset\; \prod_i J_{\ul{\BL}_i}(Q_{\nu_i})
						$$
						is a subgroup, which is discrete for the product of the $\nu_i$-adic topologies, and separated in the profinite topology, that is, for every $1\ne g \in \Gamma_{\bar\gamma}$ there is a normal subgroup of finite index in $\Gamma_{\bar\gamma}$ that does not contain $g$. In particular the closed substack of \eqref{EqSourceOfTheta} defined by the largest ideal of definition is the Deligne-Mumford stack locally of finite type over $\Spec\BaseFldInSectUnif$ given by
						\begin{equation}\label{EqReducedSourceOfTheta}
							I_{\ul\CG_0}\!(Q) \big{\backslash}\bigl(\prod_i \Breve{X}_{Z_i}(\ul\BL_i)\times X^{\ul \nu}(\ul\CG_0)/H\bigr) \enspace\cong\enspace \coprod_{\bar\gamma} \Gamma_{\bar\gamma}\big{\backslash}\prod_i \Breve{X}_{Z_i}(\ul\BL_i)\,.
						\end{equation}

					}
					
				\end{remark}

				In the following proposition we recall the various actions on the source and the target of the above map \ref{EqUnifMorphismI}, from \cite{AH_Global}  and we will further observe  that this is equivariant under these actions.

				\begin{proposition}\label{PropThetaisEquivariantUnderGZPhi}
					There are the actions of 
					\begin{enumerate}
						\item the quasi-isogeny group $I(Q)$, on the source, and
						\item the group $\prod_iZ(Q_{\nu_i})$, where $Z\subset\FG\times_C\Spec Q$ is the center, on the source and target
						\item the Frobenius $\ul\Phi_m$ (see the proof below for the construction), and finally
						\item Hecke, after passing to the limit over all level $H$-structures, through the Hecke correspondence, on the source and target of the morphism \ref{EqUnifMorphismI}. 
					\end{enumerate}
					Moreover, the actions a), b) and d) commute with each other. 
				\end{proposition}

				\begin{proof}
					
					a) \textbf{the action of the quasi-isogeny group $I(Q)$} The group $\QIsog_{\BaseFldOfLocSht}(\ul\BL)$ of quasi-isogenies of $\ul\BL$ acts on $\RZ$ via $j\colon(\ul\CL,\bar\delta)\mapsto(\ul\CL,j\circ\bar\delta)$. Since $\ul\BL=\bigl((L^+\BP)_\BaseFldOfLocSht,b\hat{\sigma}^*\bigr)$ is trivialized and decent, $\QIsog_{\BaseFldOfLocSht}(\ul\BL)=J_b(\BaseOfD\dpl z\dpr)$ where $J_b$ is the connected algebraic group over $\BaseOfD\dpl z\dpr$ which is defined by its functor of points that assigns to an $\BaseOfD\dpl z\dpr$-algebra $R$ the group
					\begin{equation}\label{EqGroupJ}
						J_b(R):=\bigl\{\,j \in \genericG(R\otimes_{\BaseOfD\dpl z\dpr} {\BaseFldOfLocSht\dpl z\dpr})\colon j^{-1}b\hat{\sigma}(j)=b\,\bigr\}\,.
					\end{equation}
					In particular, $\QIsog_{\BaseFldOfLocSht}(\ul\BL)=J_b(\BaseOfD\dpl z\dpr)\subset L\genericG(\ell)$ and this group acts on $\Breve{X}_Z(b)$ via $j\colon g\mapsto j\cdot g$ for $j\in J_b(\BaseOfD\dpl z\dpr)$. The natural action of the quasi-isogeny group $J_{\ul{\BL}_i}(Q_{\nu_i})=\QIsog_\BaseFldInSectUnif(\ul{\BL}_i)$ over $\BaseFldInSectUnif$ $\Breve{\CM}_{\ul{\BL}_i}^{\hat{Z}_i}$; induces a natuaral action  of the group $I_{\ul\CG_0}(Q)$ on $\prod_i \Breve{\CM}_{\ul{\BL}_i}^{\hat{Z}_i}$ via the natural morphism 
					\begin{equation}\label{EqI(Q)}
						I_{\ul\CG_0}(Q)\longto \prod_i J_{\ul{\BL}_i}(Q_{\nu_i}),\es \alpha\mapsto \bigl(\omega_{\nu_i}(\alpha)\bigr)_i=:(\alpha_i)_i.
					\end{equation}

					The group $I_{\ul\CG_0}\!(Q)$ also acts naturally on $\omega^\ul\nu(\ul{\CG}_0)$ and  $X^\ul\nu(\ul\cG_0)$ by sending $\gamma\in X^\ul\nu(\ul\cG_0)$ to $\omega^\ul\nu(\eta)\circ\gamma$ for $\eta\in I_{\ul\CG_0}\!(Q)$. After choosing an element $\gamma_0\in\Isom^{\otimes}(\omega_{\BO^{\ul\nu}}^\circ,\omega_{\BO^{\ul\nu}}^\ul\nu(\ul{\CG}_0))$, this defines a morphism 
					\begin{equation}\label{EqEpsilon}
						\epsilon\colon  I_{\ul\CG_0}\!(Q) \;\longto\; Aut^\otimes(\omega^\circ)\;\cong\;\FG(\BA^{\ul\nu}),\quad\eta\mapsto\gamma_0\circ\omega^\ul\nu(\eta)\circ\gamma_0^{-1}.
					\end{equation}

					\bigskip 
					
					\noindent
					b) \textbf{the action of the group $Z(Q)$} First observe that for any local $\BP_{\nu_i}$-shtuka $\ul\CL_i$ over a scheme $S\in\Nilp_{A_{\nu_i}}$, the element $c_i$ in $Z(Q_{\nu_i})$ induces an element $c_i\in\QIsog_S(\ul\CL_i)$ of the quasi-isogeny group of $\ul\CL_i$. Namely, over an \'etale covering $S'\to S$ we can choose a trivialization $\alpha\colon\ul\CL_{i,S'}\isoto\bigl((L^+\BP_{\nu_i})_{S'},b'_i\hat\sigma^*_{\nu_i}\bigr)$ and then $c_i=\hat\sigma_{\nu_i}(c_i)$ implies $c_i\,b'_i=b'_i\,\hat\sigma_{\nu_i}(c_i)$, that is $\alpha^{-1}\circ c_i\circ\alpha\in\QIsog_{S'}(\ul\CL_i)$. To show that this quasi-isogeny descends to $S$ let $pr_1,pr_2\colon S'':=S'\times_S S' \to S'$ be the projections onto the first and second factor and set $g:=pr_2^*\alpha\circ pr_1^*\alpha^{-1}\in L^+\BP_{\nu_i}(S'')$. Then $g \circ pr_1^*c_i\circ g^{-1}=pr_1^*c_i = pr_2^*c_i$ implies
					\[
					pr_1^*(\alpha^{-1}\circ c_i\circ\alpha)\;=\;pr_2^*\alpha^{-1}\circ g \circ pr_1^*c_i\circ g^{-1}\circ pr_2^*\alpha\;=\;pr_2^*(\alpha^{-1}\circ c_i\circ\alpha)\,,
					\]
					and this shows that $\alpha^{-1}\circ c_i\circ\alpha$ descends to a quasi-isogeny of $\ul\CL_i$ over $S$ which we denote by $c_i\in\QIsog_S(\ul\CL_i)$. If moreover we are given a quasi-isogeny $\delta_i\colon\ul\CL_i\to\ul\BL_{i,S}$, that is if $(\ul\CL_i,\delta_i)\in\CM_{\ul\BL_i}(S)$, then $\delta_i$ is automatically compatible with the quasi-isogenies $c_i\in\QIsog_S(\ul\CL_i)$ and $c_i\in\QIsog_\BaseFldInSectUnif(\ul\BL_i)$, that is $\delta_i\circ c_i=c_i\circ\delta_i$. Indeed, using the trivialization $\alpha$ over $S'$ from above, $\delta_i$ corresponds to $g_i:=\delta_i\circ\alpha^{-1}\in L\genericG_{\nu_i}(S')$, and then
					\[
					\delta_i\circ c_i\;:=\;\delta_i\circ(\alpha^{-1}\circ c_i\circ\alpha)\;=\;g_i\circ c_i\circ\alpha\;=\;c_i\circ g_i\circ\alpha\;=\;c_i\circ\delta_i\,.
					\]
					This shows that the action of $c_i\in Z(Q_{\nu_i})$ on $\breve\CM_{\ul{\BL}_i}^{\hat{Z}_i}$ is given as
					\begin{equation}\label{EqActionCenter3}
						c_i\colon\breve\CM_{\ul{\BL}_i}^{\hat{Z}_i}\;\longto\;\breve\CM_{\ul{\BL}_i}^{\hat{Z}_i}\,,\quad (\ul\CL_i,\delta_i)\;\longmapsto\; (\ul\CL_i,c_i\circ\delta_i)\;=\;(\ul\CL_i,\delta_i\circ c_i)\,. 
					\end{equation}
					This action commutes with the action of $\eta\in I_{\ul\CG_0}\!(Q)$, because $\eta$ and $(c_i)_i$ act on $\prod_i \breve\CM_{\ul\BL_i}^{\hat{Z}_i}\times X^\ul\nu(\ul\CG_0))/H$ by 
					\[
					(\ul{\CL}_i,\delta_i)_i \times \gamma H\;\longmapsto\;(\ul{\CL}_i,\omega_{\nu_i}(\eta)\circ\delta_i\circ c_i)_i\times \omega^\ul\nu(\eta)\gamma H\,.
					\]
					
					On the other hand $(c_i)_i\in\prod_i Z(Q_{\nu_i})$ also acts on $\nabla_n^{H,\ulHZ}\scrH^1(C,\FG)^{\ul\nu}$ as follows. Let $(\ul\CG,\gamma H)$ be in $\nabla_n^{H,\ulHZ}\scrH^1(C,\FG)^{\ul\nu}(S)$ and consider the associated local $\BP_{\nu_i}$-shtukas $\ul\CL_i:=\omega_{\nu_i}(\ul\CG)$. We have seen above that $c_i$ induces an element $c_i\in\QIsog_S(\ul\CL_i)$ of the quasi-isogeny group of $\ul\CL_i$. By Lemma~\ref{LemLocalIsogeny} there is a uniquely determined global $\FG$-shtuka $c_n^*\circ\ldots\circ c_1^*\,\ul\CG$ and a quasi-isogeny $c\colon c_n^*\circ\ldots\circ c_1^*\,\ul\CG\to\ul\CG$ which is an isomorphism outside the $\nu_i$ and satisfies $\omega_{\nu_i}(c)=c_i$. We can now define the action of $(c_i)_i\in\prod_i Z(Q_{\nu_i})$ on $(\ul\CG,\gamma H)\in\nabla_n^{H,\ulHZ}\scrH^1(C,\FG)^{\ul\nu}(S)$ as
					\begin{equation}\label{EqActionCenter2}
						(c_i)_i\colon\;(\ul\CG,\gamma H)\;\longmapsto\;(c_n^*\circ\ldots\circ c_1^*\,\ul\CG\,,\,\omega^\ul\nu(c)^{-1}\gamma H)\,.
					\end{equation}
					
					\bigskip
					\noindent
					c) \textbf{The Frobenius action $\ul\Phi_m$} First we observe that the source and target of $\Theta$ carry a \emph{Weil-descent datum} for the ring extension $R_{\ulHZ}\subset\breve R_{\ulHZ}=k\dbl\xi_1,\ldots,\xi_n\dbr$, compare~\cite[Definition~3.45]{RZ}.  Recall that for a stack $\CH$ over $\Spf\breve R_{\ulHZ}$, we consider the stack $\CH^\lambda$ defined by 
					\[
					\CH^\lambda(S)\;:=\;\CH(S_{[\lambda]}),\,
					\]
					for a scheme $(S,\theta)\in\Nilp_{\breve R_{\ulHZ}}$ where $\theta\colon S\to\Spf\breve R_{\ulHZ}$ denotes the structure morphism of the scheme $S$. Here $\lambda$ denote the $R_{\ulHZ}$-automorphism of $\breve R_{\ulHZ}$ given by
					\[
					\lambda\colon \xi_i\;\longmapsto\;\xi_i \quad\text{and}\quad \lambda|_k\;=\;\Frob_{\#\kappa,\BaseFldInSectUnif}\colon x\;\longmapsto\;x^{\#\kappa}\quad\text{for }x\in k\,
					\]
					and $S_{[\lambda]}\in\Nilp_{\breve R_{\ulHZ}}$ denote the pair ($S,\lambda\circ\theta)$.
					Then a \emph{Weil-descent datum} on $\CH$ is an isomorphism of stacks $\CH\isoto\CH^\lambda$, that is an equivalence $\CH(S)\isoto\CH(S_{[\lambda]})$ for every $S\in\Nilp_{\breve R_{\ulHZ}}$, compatible with morphisms in $\Nilp_{\breve R_{\ulHZ}}$.

					On $\nabla_n^{H,\ulHZ}\scrH^1(C,\FG)^{\ul\nu}\whtimes_{R_{\ulHZ}} \Spf\breve R_{\ulHZ}$ the canonical Weil-descent datum is given by the identity
					\begin{equation}\label{EqDescentDatumOnNablaH}
						\id\colon\nabla_n^{H,\ulHZ}\scrH^1(C,\FG)^{\ul\nu}(S)\;\isoto\;\nabla_n^{H,\ulHZ}\scrH^1(C,\FG)^{\ul\nu}(S_{[\lambda]})\,,\quad(\ul\CG,\gamma H)\;\longmapsto\;(\ul\CG,\gamma H)
					\end{equation}
					because under the inclusion $\Nilp_{\breve R_{\ulHZ}}\into\Nilp_{R_{\ulHZ}}$ we have $S=S_{[\lambda]}$ in $\Nilp_{R_{\ulHZ}}$.

					Consider the Weil descent datum given by
					\begin{align}\label{EqDescentDatumSource}
						\prod_i\Breve{\CM}_{\ul\BL_i}^{\hat{Z}_i}(S)\;\isoto\; & \prod_i\Breve{\CM}_{\ul\BL_i}^{\hat{Z}_i}(S_{[\lambda]}),\\
						(\ul\CL_i,\delta_i\colon\ul\CL_i\to\ul\BL_{i,S})_i\;\longmapsto\; & (\ul\CL_i,\theta^*(\tau_{\BL_i}^{-[\kappa\colon\BF_q]})\circ\delta_i\colon\ul\CL_i\to\ul\BL_{i,S_{[\lambda]}})_i\,.\nonumber
					\end{align}
					Here we observe that $\BL_{i,S}:=\theta^*\ul\BL_i$ and $\ul\BL_{i,S_{[\lambda]}}:=(\lambda\circ\theta)^*\ul\BL_i=\theta^*\lambda^*\ul\BL_i=\theta^*\sigma^{[\kappa\colon\BF_q]*}\ul\BL_i$, and that $\tau_{\BL_i}^{-[\kappa\colon\BF_q]}\colon\ul\BL_i\to\sigma^{[\kappa\colon\BF_q]*}\ul\BL_i$ is a quasi-isogeny.

					Let $m$ is a multiple of $[\kappa\colon\BF_q]$, we define the $q^m$-Frobenius morphism  
					\begin{align}\label{EqFrobOnSource}
						\ul\Phi_m\colon\;\bigl(\prod_i\breve\CM_{\ul\BL_i}^{\hat{Z}_i}\whtimes_{R_i}\Spec\BaseFldInSectUnif\bigr) & \times X^\ul\nu(\ul\CG_0)/H\;\longto\; \\
						&  \bigl(\prod_i\Breve{\CM}_{\ul\BL_i}^{\hat{Z}_i}\whtimes_{R_i}\Spec\BaseFldInSectUnif\bigr)\times X^\ul\nu(\ul\CG_0)/H\,.\nonumber
					\end{align}
					in the following way.

					Let $(\ul\CL_i,\delta)_i$ be a point in $\prod_i\Breve{\CM}_{\ul\BL_i}^{\hat{Z}_i}(S)$. Set $\lambda=\lambda_m=\Frob_{m,k}$. Note that since $m$ is a multiple of $[\kappa_i:\BF_q]$ and $Z_i$ is defined over $\kappa_i$, the local $\BP$-shtuka $\sigma^{m\ast}$ is bounded by $\hat{Z}_i$ and thus $(\sigma^{m\ast}(\ul\CL_i,\delta_i)_i$ lies in $\left(\prod_i\Breve{\CM}_{\ul\BL_i}^{\hat{Z}_i}\right)^{\lambda_m}(S)$, which maps to  $(\sigma^{m*}\ul\CL_i, \theta^*(\tau_{\BL_i}^m)\circ\sigma^{m*}\delta_i)_i$ in $\left(\prod_i\Breve{\CM}_{\ul\BL_i}^{\hat{Z}_i}\right)(S)$ via \ref{EqDescentDatumSource}. This defines the $q^m$-Frobenius endomorphism

					\begin{equation}\label{EqFrobOnRZ}
						\ul\Phi_m\colon\; \bigl(\prod_i\Breve{\CM}_{\ul\BL_i}^{\hat{Z}_i}\whtimes_{R_i}\Spec\BaseFldInSectUnif\bigl) \times X^\ul\nu(\ul\CG_0)/H\to \bigl(\prod_i\Breve{\CM}_{\ul\BL_i}^{\hat{Z}_i}\whtimes_{R_i}\Spec\BaseFldInSectUnif\bigr)\times X^\ul\nu(\ul\CG_0)/H\nonumber
					\end{equation}
					which is $(\ul\CL_i,\delta_i)\mapsto (\sigma^{m\ast}\ul\CL_i, \tau_{\BL_i}^m\circ\sigma^{m\ast}\delta_i)$ on the first factor and $\id$ on $X^\ul\nu(\ul\CG_0)\slash H$.

					On $\prod_i \Breve{\CM}_{\ul\BL_i}^{\hat{Z}_i}\times X^\ul\nu(\ul\CG_0)/H$ we consider the product of the Weil Descent data \eqref{EqDescentDatumSource} with the identity on $X^\ul\nu(\ul\CG_0)/H$. This Weil descent datum commutes with the action of $\eta\in I_{\ul\CG_0}\!(Q)$. Note that $\theta^*(\tau_{\BL_i}^{-[\kappa\colon\BF_q]}\circ\omega_{\nu_i}(\eta))=\theta^*(\sigma^{[\kappa\colon\BF_q]*}(\omega_{\nu_i}(\eta))\circ\tau_{\BL_i}^{-[\kappa\colon\BF_q]})=(\lambda\theta)^*(\omega_{\nu_i}(\eta))\circ\theta^*(\tau_{\BL_i}^{-[\kappa\colon\BF_q]})$. This defines a Weil descent datum on $\CY:=I_{\ul\CG_0}\!(Q) \big{\backslash}\bigl(\prod_i \Breve{\CM}_{\ul\BL_i}^{\hat{Z}_i}\times \Isom^{\otimes}(\omega^\circ,\check{\CV}_{\ul{\CG}_0})/H\bigr)$ by
					
					\begin{align}\label{EqDescentDatumSourceModI}
						\CY(S) \;\isoto\; & \CY^\lambda(S)\;=\;\CY(S_{[\lambda]}) \\
						(\ul\CL_i,\delta_i)_i\times \gamma H\;\longmapsto\; & (\ul\CL_i,\theta^*(\tau_{\BL_i}^{-[\kappa\colon\BF_q]})\circ\delta_i)_i\times \gamma H\,. \nonumber
					\end{align}

					It follows that $\ul\Phi_m$ commutes with the action of $\eta\in I_{\ul\CG_0}\!(Q)$, because the composition is given by 
					\[
					(\ul{\CL}_i,\delta_i)_i \times \gamma H\;\longmapsto\;(\ul{\CL}_i,\omega_{\nu_i}(\eta)\circ\delta_i\circ\tau_{\CL_i}^m)_i \times \omega^\ul\nu(\eta)\gamma H\,.
					\]

					\bigskip

					\noindent
					d) \textbf{Hecke correspondences and Hecke operation}
					The action of $h\in\FG(\BA^{\ul\nu})$ by Hecke correspondences is explicitly given as follows. Let $H,H'\subset\FG(\BA^{\ul\nu})$ be compact open subgroups. Then the Hecke correspondences $\pi(h)_{H'\!,H}$ are given by the diagrams
					\begin{equation}\label{EqHeckeSource}
						\xymatrix @C=-4pc {
							& \prod_i \Breve\CM_{\ul\BL_i}^{\hat{Z}_i}\times X^{\ul \nu}(\ul\CG_0)/(hHh^{-1}\cap H') \ar[dl] \ar[dr] \\
							\prod_i \Breve\CM_{\ul\BL_i}^{\hat{Z}_i}\times X^{\ul \nu}(\ul\CG_0)/H& & \prod_i \Breve\CM_{\ul\BL_i}^{\hat{Z}_i}\times X^{\ul \nu}(\ul\CG_0)/H'\ar@{-->}[ll]
						}
					\end{equation}
					where the left (resp. right) hand side morphism in the above roof sends $(\ul{\CL}_i,\delta_i)_i \times \gamma(hHh^{-1}\cap H')$ to $(\ul{\CL}_i,\delta_i)_i \times \gamma hH$ (resp. $(\ul{\CL}_i,\delta_i)_i \times \gamma H'$),
					and
					\begin{equation}\label{EqHeckeTarget}
						\xymatrix @C=0pc {
							& \nabla_n^{(hHh^{-1}\cap H'),\ulHZ}\scrH^1(C,\FG)^{\ul\nu} \ar[dl] \ar[dr] \\
							\nabla_n^{H,\ulHZ}\scrH^1(C,\FG)^{\ul\nu} & & \nabla_n^{H'\!,\ulHZ}\scrH^1(C,\FG)^{\ul\nu} \ar@{-->}[ll],
						}
					\end{equation}
					where the left (resp. right) hand side morphism in the above roof sends $(\ul\CG,\gamma(hHh^{-1}\cap H'))$ to $(\ul\CG,\gamma hH)$ (resp. $(\ul\CG,\gamma H')$). From the definition of the $H$-level structure and the morphism $\Theta$ it is clear that the map $\Theta$ is equivariant under the Hecke action

					A special case for $H'\subset H$ and $h=1$ are the forgetful morphisms 
					\[
					\pi(1)_{H'\!,H}\colon\prod_i \Breve\CM_{\ul\BL_i}^{\hat{Z}_i}\times X^\ul\nu(\ul\cG_0)/H'\to\prod_i \Breve\CM_{\ul\BL_i}^{\hat{Z}_i}\times X^\ul\nu(\ul\cG_0)/H
					\]
					and $\pi(1)_{H'\!,H}\colon\nabla_n^{H'\!,\ulHZ}\scrH^1(C,\FG)^{\ul\nu}\to\nabla_n^{H,\ulHZ}\scrH^1(C,\FG)^{\ul\nu}$, which are finite \'etale and surjective; see \cite[Theorem 6.7 b)]{AH_Global}.

				\end{proof}

			\end{point}

			We return to the proof of theorem \ref{ThmL-RForG-ShtTXT}. Since the above actions, Proposition \ref{PropThetaisEquivariantUnderGZPhi}, commute with the action of the quasi-isogeny group $I(Q)$, the morphism \ref{EqUnifMorphismI} induces the following morphism
			
			$$
			\Theta_{\ul{\cG}_0}\colon  I_{\ul{\cG}_0}(Q) \big{\backslash}\prod_i \Breve\CM_{\ul\BL_i}^{\hat{Z}_{\nu_i}}\times X^\ul\nu(\ul\cG_0) {\slash}H \longto \nabla_n^{H,\hat{Z}_{\ul\nu}}\scrH^1(C,\FG)^{{\ul\nu}}\whtimes_{\BF_{\ul\nu}}\Spec\BaseFldInSectUnif
			$$
			of ind-algebraic stacks over $\Spf\BaseFldInSectUnif\dbl\ul\zeta\dbr$ which is ind-proper and formally \'etale, and moreover is equivariant under the actions mentioned in Proposition \ref{PropThetaisEquivariantUnderGZPhi}. Now let $\{T_\iii\}$ be a set of representatives of $I(Q)$-orbits of the irreducible components of $\prod_i \Breve\CM_{\ul{\BL}_i}^{\hat{Z}_i}\times \FG(\BA_Q^{\ul\nu}) {\slash}H$. Then the image $\Theta_{\ul\CG_0}'(T_\iii)$ of $T_\iii$ under $\Theta_{\ul\CG_0}'$ is closed and each $\Theta_{\ul\CG_0}'(T_\iii)$ intersects only finitely many others. Let $\CZ$ denote the union of the $\Theta_{\ul\CG_0}'(T_\iii)$ and let $\nabla_n^{H,\hat{Z}_{\ul\nu}},\scrH^1(C,\FG)_{/\CZ}^{{\ul\nu}}$ be the formal completion of $\nabla_n^{H,\hat{Z}_{\ul\nu}},\scrH^1(C,\FG)^{{\ul\nu}}\whtimes_{\BF_{\ul\nu}}\Spec\BaseFldInSectUnif$ along $\CZ$. According to \cite[Theorem 7.11]{AH_Global}, $\Theta_{\ul\CG_0}$ induces an isomorphism of formal algebraic stacks over $\Spf\BaseFldInSectUnif\dbl\ul\zeta\dbr$
			$$
			\Theta_{{\ul\CG_0}_{/_\CZ}}\colon  I_{\ul\CG_0}(Q) \big{\backslash}\prod_i \Breve\CM_{\ul\BL_i}^{\hat{Z}_{\nu_i}}\times X^\ul\nu(\ul\CG_0) {\slash}H \;\isoto\; \nabla_n^{H,\hat{Z}_{\ul\nu}}\scrH^1(C,\FG)_{/\CZ}^{{\ul\nu}}.
			$$

			By Proposition \ref{PropThetaisEquivariantUnderGZPhi} we observe that $\Theta_{{\ul\CG_0}_{/_\CZ}}$ is equivariant under the action of  $\FG(\BA^{\ul\nu})\times Z(Q)$. Moreover $\ul\Phi_m$ defines the $q^m$-Frobenius endomorphism $\ul\Phi_m$ of $I_{\ul\CG_0}\!(Q) \big{\backslash}\bigl(\prod_i \Breve{\CM}_{\ul\BL_i}^{\hat{Z}_i}\times X^\ul\nu(\ul\CG_0)/H\bigr)$.

			We must also check that the action $\ul\Phi_m$ on the left corresponds to the $m$-Frobenius $\sigma^m$ on the right, i.e. 
			
			\begin{lemma}
				The following diagram 
				
				$$
				\xymatrix @C+2pc {
					I_{\ul\CG_0}(Q) \big{\backslash}\prod_i \Breve\CM_{\ul\BL_i}^{\hat{Z}_{\nu_i}}\times X^\ul\nu(\ul\CG_0) {\slash}H \ar[r]^{\Theta_{{\ul\CG_0}_{/_\CZ}}} \ar[d]_{\ul\Phi_m} & \nabla_n^{H,\hat{Z}_{\ul\nu}}\scrH^1(C,\FG)_{/\CZ}^{{\ul\nu}} \ar[d]_{\sigma^m} \\
					I_{\ul\CG_0}(Q) \big{\backslash}\prod_i \Breve\CM_{\ul\BL_i}^{\hat{Z}_{\nu_i}}\times X^\ul\nu(\ul\CG_0) {\slash}H \ar[r]^{\Theta_{{\ul\CG_0}_{/_\CZ}}} & \nabla_n^{H,\hat{Z}_{\ul\nu}}\scrH^1(C,\FG)_{/\CZ}^{{\ul\nu}}.
				}
				$$
				
				commutes.
			\end{lemma}
			
			\begin{proof}
				We recall the statement from \cite{AH_Global}. First notice that, as we will see below,$\nabla_n^{H,\ulHZ}\scrH^1(C,\FG)^{\ul\nu}_{/\CZ}$ carries the $q^m$-Frobenius endomorphism $\sigma^m$ induced from 
				
				\begin{align}\label{EqFrobOnTarget}
					\nabla_n^{H,\ulHZ}\scrH^1(C,\FG)^{\ul\nu}\whtimes_{R_{\ulHZ}} \Spec\BaseFldInSectUnif\;\longto\; & \nabla_n^{H,\ulHZ}\scrH^1(C,\FG)^{\ul\nu}\whtimes_{R_{\ulHZ}} \Spec\BaseFldInSectUnif \nonumber \\
					(\ul\CG,\gamma H) \;\longmapsto\; & (\sigma^{m*}\ul\CG,\sigma^{m*}(\gamma)H)\,.
				\end{align}
				Since $m$ is a multiple of $[\kappa_i\colon\BF_q]$ and the reduced subscheme $Z_i$ of the bound $\hat{Z}_i$ is defined over $\kappa_i$ the Frobenius $\omega_{\nu_i}(\tau_{\sigma^{m*}\CG})=\sigma^{m*}\omega_{\nu_i}(\tau_{\CG})$ lies in $\sigma^{m*}Z_i=Z_i$. This means that also $\sigma^{m*}\ul\CG$ is bounded by $\ulHZ$.

				Suppose $\Theta_{\ul\CG_0}'(\ul\Phi_m)=\sigma^m \Theta_{\ul\CG_0}'$. Since $\ul\Phi_m$ commutes with the action of $I_{\ul\CG_0}\!(Q)$, see Proposition \ref{PropThetaisEquivariantUnderGZPhi}, this also proves that $\Theta_{\ul\CG_0}(\ul\Phi_m)=\sigma^m\Theta_{\ul\CG_0}$. Notice that if a morphism $S_\red\to\nabla_n^{H,\ulHZ}\scrH^1(C,\FG)^{\ul\nu}\whtimes_{R_{\ulHZ}} \Spf\breve R_{\ulHZ}$ given by $(\ul\CG,\gamma H)$ factors through $\CZ=\im(\Theta_{\ul\CG_0}')$, then also the morphism $S_\red\to\nabla_n^{H,\ulHZ}\scrH^1(C,\FG)^{\ul\nu}\whtimes_{R_{\ulHZ}} \Spf\breve R_{\ulHZ}$ given by $(\sigma^{m*}\ul\CG,\sigma^{m*}(\gamma)H)$ factors through $\CZ$. This shows that also $\Theta_{{\ul\CG_0}_{/\CZ}}$ is compatible with the $q^m$-Frobenius endomorphisms. So it remains to check that the following diagram 
				
				$$
				\xymatrix @C+2pc {
					\prod_i \Breve\CM_{\ul\BL_i}^{\hat{Z}_{\nu_i}}\times X^\ul\nu(\ul\CG_0) {\slash}H \ar[r]^{\Theta_{{\ul\CG_0}}'} \ar[d]_{\ul\Phi_m} & \nabla_n^{H,\hat{Z}_{\ul\nu}}\scrH^1(C,\FG)^{{\ul\nu}} \ar[d]_{\sigma^m} \\
					\prod_i \Breve\CM_{\ul\BL_i}^{\hat{Z}_{\nu_i}}\times X^\ul\nu(\ul\CG_0) {\slash}H \ar[r]^{\Theta_{{\ul\CG_0}}'} & \nabla_n^{H,\hat{Z}_{\ul\nu}}\scrH^1(C,\FG)^{{\ul\nu}}
				}
				$$
				commutes.

				Let $\ul y:=(\ul\CL_i,\delta_i)_i\times \gamma H$ in $\bigl(\prod_i\breve\CM_{\ul\BL_i}^{\hat{Z}_i}\whtimes_{R_i}\Spec\BaseFldInSectUnif\bigr) \times X^\ul\nu(\ul\CG_0)/H$ be an $S$-valued point. The image of this point in $\nabla_n^{H,\ulHZ}\scrH^1(C,\FG)^{\ul\nu}$ is given by 
				$$
				\Theta_{\ul\CG_0}'(\ul y)=(\ul\CG,\omega^\ul\nu(\delta)^{-1}\gamma H)
				$$
				where $\ul\CG:=\delta_n^*\circ\ldots\circ\delta_1^*\,\ul\CG_{0,S}$ and $\delta\colon\ul\CG\to\ul\CG_{0,S}$ is the quasi-isogeny with $\omega_{\nu_i}(\delta)=\delta_i$, which is isomorphism outside $\ul\nu$.
				The image of $\ul\Phi_m(\ul y)=(\sigma^{m*}\ul\CL_i, \tau_{\BL_i}^m\circ\sigma^{m*}\delta_i)_i\times\gamma H$ in $\nabla_n^{H,\ulHZ}\scrH^1(C,\FG)^{\ul\nu}$ is given by  
				$$
				\Theta_{\ul\CG_0}'\ul\Phi_m(\ul y)=(\ul\CG',\omega^\ul\nu(\delta')^{-1}\gamma H),
				$$ 
				where $\ul\CG':=(\tau_{\BL_n}^m\sigma^{m*}\delta_n)^*\circ\ldots\circ(\tau_{\BL_1}^m\sigma^{m*}\delta_1)^*\,\ul\CG_{0,S}$, and $\delta'\colon\ul\CG'\to\ul\CG_{0,S}$ is the quasi-isogeny with $\omega_{\nu_i}(\delta')=\tau_{\BL_i}^m\circ\sigma^{m*}\delta_i$, which is isomorphism outside $\ul\nu$.\\ 
				
				We obtain for the image $\ul\Phi_m\circ\Theta_{\ul\CG_0}'(\ul y)=(\sigma^{m*}\ul\CG,\sigma^{m*}(\omega^\ul\nu(\delta)^{-1}\gamma) H)$, which comes with the quasi-isogeny $\sigma^{m*}(\delta)\colon\sigma^{m*}\ul\CG\to\sigma^{m*}\ul\CG_{0,S}$. Then $\phi:=\sigma^{m*}(\delta)^{-1}\circ\tau_{\CG_0}^{-m}\circ\delta'\colon\ul\CG'\to\sigma^{m*}\ul\CG$ is a quasi-isogeny with $\omega_{\nu_i}(\phi)=\id_{\sigma^{m*}\ul\CL_i}$ and $\omega^\ul\nu(\phi)\circ\omega^\ul\nu(\delta')^{-1}\gamma H=\omega^\ul\nu(\sigma^{m*}\delta)^{-1}\circ\omega^\ul\nu(\tau_{\CG_0}^m)^{-1}\circ\gamma H=\sigma^{m*}(\omega^\ul\nu(\delta)^{-1}\gamma) H$. This is because $\omega^\ul\nu(\tau_{\CG_0}^{m})^{-1}\circ\gamma=\sigma^{m*}(\gamma)$. 
				The Lemma now follows.
			\end{proof}

			\noindent
			The following lemma shows that the map $\Theta_{\ul\CG_0}$ covers the Newton stratum.

			\begin{lemma}\label{LemmaThetaCoversNewtonStratum}
				Let $\ul\cG_0$, $\ul\cG_0'$ be $\FG$-shtukas in $\nabla_n\scrH^1(C,\FG)^{\ul\nu}(\BF)$. Then the following are equivalent
				
				\begin{enumerate}
					\item\label{LemmaThetaCoversNewtonStratum(a)}
					$\im \Theta_{\ul \cG_0}\cap\im \Theta_{\ul \cG_0'}\neq \emptyset $
					\item\label{LemmaThetaCoversNewtonStratum(b)}
					$\im \Theta_{\ul \cG_0}=\im \Theta_{\ul \cG_0'}$
					\item\label{LemmaThetaCoversNewtonStratum(c)}
					There is a quasi-isogeny $\ul \cG_0 \to \ul \cG_0'$ over $\BF$.
				\end{enumerate}
			\end{lemma}
			
			\begin{proof}
				\ref{LemmaThetaCoversNewtonStratum(a)} $\Rightarrow$ \ref{LemmaThetaCoversNewtonStratum(b)} obvious,
				\ref{LemmaThetaCoversNewtonStratum(b)}$\Rightarrow$ \ref{LemmaThetaCoversNewtonStratum(c)} follows from the fact that the image of the morphism $\Theta_{\ul\cG_0}$ lies in the quasi-isogeny locus of $\ul\cG_0$. It remains to show that \ref{LemmaThetaCoversNewtonStratum(c)} implies \ref{LemmaThetaCoversNewtonStratum(a)}.\\ 
				Let $f:\ul\cG_0 \to \ul\cG_0'$ be a quasi-isogeny over $\BF$. The quasi-isogeny $f$ induces an n-tuple of quasi-isogenies $\omega_{\nu_i}(f):\ul{\CL}_{\nu_i} \to \ul{\CL}_{\nu_i}'$ of local $\BP_{\nu_i}$-shtukas and also a tensor isomorphism $\omega^\ul\nu(f):\omega^\ul\nu(\ul \cG_0)\isoto\omega^\ul\nu(\ul \cG_0')$. These data in turn define the vertical arrows which make the following
				
				\[
				\xygraph{
					!{<0cm,0cm>;<1cm,0cm>:<0cm,1cm>::}
					!{(0,0) }*+{\nabla_n^{H,\hat{Z}_{\ul\nu}}\scrH^1(C,\FG)^{{\ul\nu}}\whtimes_{\BF_{\ul\nu}}\Spec\BaseFldInSectUnif}="a"
					!{(-6,0) }*+{\prod_i \Breve{X}_{Z_{\nu_i}}(\ul{\CL}_{\nu_i})\times X^{\ul \nu}(\ul\cG_0) {\slash} H}="c"
					!{(-6,-3) }*+{\prod_i \Breve{X}_{Z_{\nu_i}}(\ul{\CL}_{\nu_i}')\times X^{\ul \nu}(\ul\cG_0'.) {\slash} H}="e"
					"c":^{\Theta_{\ul \cG_0}'}"a"
					"c":@/_1em/"e"
					"e":"c"
					"e":^{\Theta_{\ul \cG_0'}'}"a"
				} 
				\]

				commutative. 
				
			\end{proof}
			
			After passing to the special fiber and taking the disjoint union over quasi-isogeny classes of $\FG$-shtukas, we obtain
			
			$$
			\bigsqcup_\ul\CG I_{\ul\CG}(Q)\backslash \prod_i \Breve{X}_{Z_{\nu_i}}(\ul{\CL}_{\nu_i})\times X^{\ul \nu}(\ul\cG)) {\slash} H\to \nabla_n^{H,\hat{Z}_{\ul\nu}}\scrH^1(C,\FG)^{{\ul\nu}}\whtimes_{\BF_{\ul\nu}}\Spec\BaseFldInSectUnif
			$$

			Note that according to Remark \ref{RemarkRealizationOfGMotives}, Remark \ref{RemarkEquivOFMotCats} and Corollary \ref{CorQIsogEqAut}, the left hand side can be described equivalently  as
			
			\begin{equation}
				\bigsqcup_\phi \Aut(\phi) \big{\backslash}\prod_{\nu_i} \Breve{X}_{Z_{\nu_i}}(\omega_{\nu_i}(\phi))\times X^{\ul \nu} (\phi)
				{\slash} H,
			\end{equation}
			
			\noindent
			where $\phi$ runs over isomorphism classes $\phi$ in $\FG\text{-}\breve{\CM ot} (\BF_q^{\alg})$ and $X^{\ul \nu}(\phi)=\Isom^{\otimes}(\omega^\circ,\omega^\ul\nu(\phi))$, see Definition \ref{ADLV}. Taking limit over all compact open subgroups $H$, we obtain 
			
			\begin{equation*}
				\CL^*(\FG,\hat{Z}_\ul\nu)\tilde{\to} \nabla_n^{*,\hat{Z}_\ul\nu}\scrH^1(C,\FG)^{\ul\nu}(\ol \BF).
			\end{equation*}
			This isomorphism is compatible with the action of $\FG(\BA_Q^{\ul\nu})\times Z(Q)\times\ul\Phi$, see Proposition \ref{PropThetaisEquivariantUnderGZPhi}. 
			Note that, when $$\hat{Z}_{\ul \nu}:=(\CS(w_i)\wh \times_{\BF_{\omega_i}}\Spf \BF_{\omega_i}\dbl \zeta \dbr)_i$$ then there is a group theoretic description for the admissiblity according to \ref{PropHE}. Moreover, by \ref{ThmHamacher-Kim} and Lemma \ref{LemPSHtLocTrivForEtTop} the quasi-isogeny class of a $\FG$-shtuka $\ul\CG$ always contains a special one. 
			
		\end{proof}


		%
		%

		\vfill
		
		\begin{minipage}[t]{0.5\linewidth}
			\noindent
			Esmail Arasteh Rad,\\ Institute for research in \\
			fundamental sciences (IPM)\\ P. O. Box 19395-5746 Tehran\\ Iran\\ email: \href{earasteh@ipm.ir}{earasteh@ipm.ir}
			
		\end{minipage}
		\begin{minipage}[t]{0.45\linewidth}
			\noindent
			Urs Hartl\\
			Universit\"at M\"unster\\
			Mathematisches Institut \\
			Einsteinstr.~62\\
			D -- 48149 M\"unster
			\\ Germany
			\\[1mm]
			\href{http://www.math.uni-muenster.de/u/urs.hartl/index.html.en}{www.math.uni-muenster.de/u/urs.hartl/}
		\end{minipage}
		
	\end{document}